\documentclass[a4paper,10pt]{amsart}

\usepackage[latin1]{inputenc}
\usepackage{amsmath, amsthm, amssymb, comment}
\usepackage[dvips]{graphicx}
\usepackage{enumitem}
\usepackage[colorlinks=true]{hyperref}
\usepackage{thm-restate}

\usepackage[usenames,dvipsnames]{pstricks}
\usepackage{epsfig}
\usepackage{pst-grad} 
\usepackage{pst-plot} 
\usepackage{pstricks-add}
\usepackage{pst-solides3d}

\usepackage[margin=3.5cm]{geometry}

\usepackage{pinlabel}

\usepackage{color}

\newtheorem{theorem}{Theorem}[section]
\newtheorem{corollary}[theorem]{Corollary}
\newtheorem{proposition}[theorem]{Proposition}
\newtheorem{definition}[theorem]{Definition}
\newtheorem{lemma}[theorem]{Lemma}
\newtheorem{claim}[theorem]{Claim}

\newtheorem*{theorem*}{Theorem}
\newtheorem*{proposition*}{Proposition}
\newtheorem*{definition*}{Definition}
\newtheorem*{lemma*}{Lemma}
\newtheorem*{claim*}{Claim}
\newtheorem*{corollary*}{Corollary}
\newtheorem*{convention*}{Convention}
\newtheorem{observation}[theorem]{Observation}

\theoremstyle{definition}
\newtheorem{convention}[theorem]{Convention}

\newtheorem{question}{Question}

\theoremstyle{remark}
\newtheorem{rem}[theorem]{Remark}
\newtheorem*{rem*}{Remark}
\newtheorem*{acknowledgement}{Acknowledgement}

\newcommand{\wt}[1]{\widetilde{#1}}

\newcommand\bR{\mathbb R}

\newcommand\bZ{\mathbb Z}
\newcommand\bN{\mathbb N}

\newcommand\bH{\mathbb H}

\newcommand{\R}{\mathbb R}

\DeclareMathOperator{\disp}{d}

\DeclareMathOperator{\trans}{t}
\DeclareMathOperator{\Homeo}{Homeo}

\newcommand\orb{ \mathcal O }

\newcommand\hfs{\widetilde{\mathcal F}^{s} }

\newcommand\hfu{\widetilde{\mathcal F}^{u} }

\newcommand{\cC}{\mathcal{C}}
\newcommand{\cD}{\mathcal{D}}
\newcommand{\cP}{\mathcal{P}}

\newcommand{\cT}{\mathcal{T}}

\newcommand{\MCG}{\mathrm{MCG}}

\newcommand{\HomeoZ}{\mathrm{Homeo}^{\bZ}(\bR)}

\newcommand{\fix}{\mathrm{Fix}}

\title[Orbit equivalence of $\bR$-covered Anosov flows]{Orbit equivalence of $\bR$-covered Anosov flows and hyperbolic-like actions on the line}

\author[Thomas Barthelm\'e]{Thomas Barthelm\'e}
\address{Queen's University, Kingston, Ontario}
\email{thomas.barthelme@queensu.ca}
\urladdr{sites.google.com/site/thomasbarthelme}

\author[Kathryn Mann]{Kathryn Mann, appendix with Jonathan Bowden}
 \address{Cornell University, Ithaca, NY}
 \email{k.mann@cornell.edu}
\urladdr{https://e.math.cornell.edu/people/mann}

\begin{document}
 
 \begin{abstract}
We prove a rigidity result for group actions on the line whose elements have what we call ``hyperbolic-like" dynamics.  Using this, 
we give a rigidity theorem for $\bR$-covered Anosov flows on 3-manifolds, characterizing orbit equivalent flows in terms of the elements of the fundamental group represented by periodic orbits.  As consequences of this, we give an efficient criterion to determine the isotopy classes of self orbit equivalences of $\bR$-covered Anosov flows, and prove finiteness of contact Anosov flows on any given manifold.

In the appendix with Jonathan Bowden, we prove that orbit equivalences of contact Anosov flows correspond exactly to isomorphisms of the associated contact structures. This gives a powerful tool to translate results on Anosov flows to contact geometry and vice versa.  We illustrate its use by giving two new results in contact geometry: the existence of manifolds with arbitrarily many distinct Anosov contact structures, answering a question of Foulon--Hasselblatt--Vaugon, and a virtual description of the group of contact transformations of a  Anosov contact structure, generalizing a result of Giroux and Massot.
 \end{abstract}
 
 \maketitle

\section{Introduction}

\subsection{Hyperbolic-like actions} 
A well known theorem of H\"older states that any group acting freely by homeomorphisms of the line is abelian and conjugate to a group of translations. This was generalized in an unpublished work of Solodov (see, e.g.~\cite{FF, Kovacevic, Bar_caracterisation}) to the statement that a group action on the line where each nontrivial element has at most {\em one} fixed point is semi-conjugate to an action by affine transformations, or abelian with a global fixed-point.  Later, the proof of the Convergence Group Theorem \cite{Gabai, CassonJungreis94} established that a group action on the circle where each element has at most {\em two} fixed points is, under some additional technical dynamical hypotheses, conjugate to a subgroup of  $\mathrm{PSL}(2,\R)$ acting on $\R \mathrm{P}^1$ by M\"obius transformations.  This important result was the last step in the proof of the Seifert fiber space conjecture. 

While one cannot reasonably expect further generalizations in this vein\footnote{One reason for this is that the target groups $\R$, $\mathrm{Aff_+(\R)}$, and $\mathrm{PSL}(2,\R)$ are essentially the only Lie groups acting transitively on $1$-manifolds, so the only natural candidates for such targets.}, it is a natural question to ask what other fixed-point data might determine an action. In this spirit, we show that under suitable hypotheses, an action of a group on the line is determined up to conjugacy by the set of elements acting with fixed points.   Like the statement of the convergence group theorem, our hypotheses are motivated by an application to a classification problem, in our case the classification of $\R$-covered Anosov flows on $3$-manifolds.  Say that an action on the line is {\em hyperbolic-like} if it commutes with integer translation, and each nontrivial element either acts freely or has exactly two fixed points in $[0,1)$, one attracting and one repelling.  We prove the following rigidity result for such actions. 

\begin{restatable}[Rigidity of hyperbolic-like actions]{theorem}{hypthm} \label{thm:hyp_action}
A minimal, hyperbolic-like action of a nonabelian group $G$ on $\R$ is determined up to conjugacy by the set of elements of $G$ that act with fixed points.    
\end{restatable} 

Minimal and hyperbolic-like are both properties of the actions of $3$-manifold fundamental groups on $\R$ induced by $\R$-covered Anosov flows.  This allows us to use Theorem \ref{thm:hyp_action} to classify such flows up to orbit equivalence.  

\subsection{Orbit equivalence of Anosov flows} 
Recall that two flows on a manifold $M$ are {\em orbit equivalent} if there is a homeomorphism $f\colon M \to M$ taking orbits of one to orbits of the other, and {\em isotopically equivalent} if this homeomorphism can be taken to be isotopic to the identity.  

An Anosov flow is called $\R$-covered if the leaf space of its weak-stable foliation is homeomorphic to $\R$. (In the case of 3-manifolds, it is equivalent that the weak-unstable foliation has leaf space $\R$).  On 3-manifolds, there are many constructions of $\R$ covered flows, and examples of manifolds admitting arbitrarily many inequivalent $\R$-covered flows. 
Here, we give a characterization of orbit and isotopy equivalent $\R$-covered flows on $3$-manifolds by their {\em free homotopy classes of periodic orbits}.  
For a flow $\varphi$ on a manifold $M$, let $\cP(\varphi)$ denote the set of conjugacy classes of elements in $\pi_1(M)$ represented by the free homotopy classes of periodic orbits of $\varphi$.  We show the following.  

\begin{theorem}[Classification of $\bR$-covered Anosov flows] \label{thm_main}
Let $\varphi$ and $\psi$ be $\bR$-covered Anosov flows on a closed $3$ manifold $M$.  
\begin{enumerate} 
\item $\varphi$ and $\psi$  are isotopically equivalent if and only if $\cP(\varphi) = \cP(\psi)$.
\item $\varphi$ and $\psi$  are orbit equivalent\footnote{
In our definition of orbit equivalence, we do not require the homeomorphism to match \emph{oriented} orbits to oriented orbits.
If one wants to consider only orbit equivalences that preserve orbit orientation, then the conclusion of Theorem \ref{thm_main} will be that if one flow is transversally orientable then the other is also, and the orbit equivalence can be upgraded to an orientation-preserving orbit equivalence (using \cite[Th\'eor\`eme C]{Bar_caracterisation}).  If one (hence both) of the flows are not transversally orientable, then $\varphi$ is orbit equivalent to either $\psi$ or $\psi^{-1}$.} if and only if there exists a homeomorphism $f\colon M \to M$ such that $f_\ast(\cP(\varphi)) = \cP(\psi)$. Moreover, the orbit equivalence can be taken to be in the isotopy class of $f$. 
\end{enumerate}
 \end{theorem}

The $\bR$-covered Anosov flows on closed 3-manifolds form a rich class of examples, including geodesic flows on closed surfaces, all contact Anosov flows, and diverse examples on many hyperbolic 3-manifolds and on manifolds with nontrivial JSJ decomposition (see~\cite{Fen_Anosov_flow_3_manifolds,FouHass_contact,BF_counting,BI,BowdenMann}).  Among other applications, our main theorem allows us to prove finiteness of contact Anosov flows.   We describe and motivate the main applications now.  

\subsection{Application 1: Classifying self orbit equivalences} \label{sec_intro_SOE}
Describing the centralizer of a given diffeomorphism is a classical question in discrete-time dynamical systems; notably, Smale's conjecture \cite{Smale} is that the centralizer of a generic diffeomorphism should be trivial.   By contrast, diffeomorphisms which embed in a flow have an $\R$-subgroup in their centralizer given by the flow, so the right analog of  Smale's conjecture in this case is to ask whether the flow agrees (virtually) with the centralizer of the diffeomorphism.  This motivates the general program to classify all symmetries of a given flow, and, more generally, classify the symmetries of the foliation given by orbits of a flow, i.e.~classify the self orbit equivalences.  
As with Smale's conjecture, this question is quite sensitive to regularity --- for instance, 3-dimensional Anosov flows often have many self orbit equivalences, while the set of those which may be realized by $C^1$ diffeomorphisms was shown to be virtually trivial by Bartheml\'e, Fenley and Potrie \cite{BFP_CAF}.

Theorem \ref{thm_main} gives the following immediate characterization of isotopy classes of self orbit equivalences. 
 \begin{corollary}\label{cor_soe_iff_preserve}
 A map $f\colon M \to M$ is in the isotopy class of a self orbit equivalence of a $\bR$-covered Anosov flow $\varphi$ if and only if $f_\ast \colon \pi_1(M) \to \pi_1(M)$ preserves the set of conjugacy classes realized by periodic orbits of $\varphi$.
 \end{corollary}

With more work, we improve this to give a criterion to explicitly describe such classes, as follows.  The case of interest here is for {\em skew} flows, those which are not orbit equivalent to a suspension of an Anosov diffeomorphism on the torus, as the orbit-equivalences of suspension flows are essentially trivial.

If $M$ is a $3$-manifold with a skew-Anosov flow $\varphi$, then $M$ is orientable and irreducible so admits a JSJ decomposition along tori into Seifert and atoroidal pieces.
By Mostow rigidity and the structure of mapping class groups of Seifert spaces, the group $\mathrm{Dehn}(M)$ of isotopy classes of diffeomorphisms generated by Dehn twists along embedded tori in $M$ has finite index in $\MCG(M)$ \cite{JoBook}.  
We give a criterion for when maps generated by certain Dehn twists represent a self orbit equivalence of a flow.   
Given a flow $\varphi$, and Dehn twist $D_{\beta}$ we may define an {\em orbit displacement} function, as follows.  
For each periodic orbit $c$ in $M$, we set
\[
\disp_\phi(c,D_\beta) = \sum_{i=1}^{k} 2 \epsilon_i  \trans_\varphi(\beta_i)
\]
where $k$ is the number of transverse intersections of $c$ with $T$ (when $T$ is in quasi-transverse position with respect to the flow,
 $\trans_\varphi(\beta)$ is the {\em translation number} of the action of $\beta$ on $\Lambda^s(\varphi)$, and $\epsilon_i = \pm 1$ is an orientation term for each intersection.  When $f$ is a composition of Dehn twists $D_{\beta_i}$ on disjoint non-isotopic tori, we define $\disp_\phi(c,f)$ to be the sum of the displacements $\disp_\phi(c,D_{\beta_i})$.  Formal definitions are given in 
Section \ref{sec:SOE}.  

\begin{theorem}[Criterion for self-orbit equivalence]\label{thm_criterion}
Let $\varphi$ be a transversally oriented skew Anosov flow.
A map $f$ which is a composition of Dehn twists on disjoint non-isotopic tori is isotopic to a self orbit equivalence of $\varphi$ if and only if, for all periodic orbits $c$ of $\varphi$, we have $\disp_\varphi(c,f) = 0$.
\end{theorem}

There are many situations in which this criterion is easy to check.  As two sample applications, we have the following results for any transversally oriented skew Anosov flow $\varphi$.  

\begin{corollary}\label{cor_horizontal_twists}
 Let $D_{\mathrm{per}}$ be the subgroup of the mapping class group of $M$ generated by Dehn twists along curves represented by periodic orbits. Then any isotopy class $[h]\in D_{\mathrm{per}}$ is represented by a self orbit equivalence of $\varphi$.
\end{corollary}

\begin{corollary}\label{cor_non_horizontal_twists}
 Let $T$ be an embedded torus and $D_{\beta}$ a Dehn twist on $T$ with nonzero translation number.  
 Then $D_{\beta}$ is isotopic to a self orbit equivalence of $\varphi$ if and only if $T$ is separating in $M$.
\end{corollary}

In Section \ref{sec:special_cases} we discuss a number of other special cases where one may use the topology of $M$ to reduce Theorem \ref{thm_criterion} to a simpler statement.  
We also discuss the complementary case to Theorem \ref{thm_criterion} for maps generated by Dehn twists in tori which cannot be realized disjointly, namely tori inside a single Seifert piece. See Theorems \ref{thm_oneJSJpiece} and Theorems \ref{thm_linear_graph}.

 \begin{rem}
 One can easily describe all self orbit equivalences of a given skew-Anosov flow in a fixed isotopy class.  Such a flow comes with the data of a homeomorphism $\eta\colon M \to M$ realizing the {\em half-step-up} map
 on the orbit space. See section \ref{sec_skew} for details. It follows from  \cite[Theorem 1.1]{BarG} that any two self orbit equivalences $h_1$ and $h_2$ in the same isotopy class differ, up to isotopy along the flow lines, by some power of $\eta$.  
 \end{rem} 
 
Theorem \ref{thm_criterion} identifies which elements of a large subgroup of the mapping class group of $M$ are represented by self-orbit equivalences.  However, passing to the full mapping class group requires a different approach.  In particular, we do not know the answer to the following.

\begin{question} 
Does there exist an ($\bR$-covered or not) Anosov flow on a hyperbolic $3$-manifold $M$ such that the only self orbit equivalences are isotopic to identity?
\end{question}
  
\begin{question} 
Does there exist an ($\bR$-covered or not) Anosov flow on a hyperbolic $3$-manifold $M$ such that every element of the mapping class group of $M$ is represented by a self orbit equivalence?
 \end{question}
 
 \begin{rem}
  The case of most interest for both questions is when the manifold considered has non-trivial mapping class group. Note however that there are, yet, no constructions of Anosov flows on a $3$-manifold that has a trivial mapping class group. So even the trivial case for these questions is not yet known.
 \end{rem}

 \begin{rem}
 In \cite{BFP_CAF}, the authors introduced a class of partially hyperbolic diffeomorphisms called ``collapsed Anosov flows", which are semi-conjugate to self orbit equivalences of Anosov flows. Hence, the criterion of Theorem \ref{thm_criterion} (and its applications for certain manifolds, as in Theorems \ref{thm_oneJSJpiece} and \ref{thm_linear_graph}) describes the possible isotopy classes of collapsed Anosov flows associated with $\bR$-covered Anosov flows. 
\end{rem}

 \subsection{Application 2: Contact Anosov flows} 
An Anosov flow is said to be \emph{contact} if it is the Reeb flow of a contact $1$-form $\alpha$.
Notice that with this definition, the contact structure $\xi = \ker \alpha$ is automatically transversely orientable since it is given as the kernel of a (globally defined) contact form. 
The contact Anosov flows are an important and well studied class of examples, as they can be thought of as a generalization of the geodesic flow on manifolds of negative curvature, and many dynamical results on Anosov flows, for instance, exponential decay of correlations \cite{Liverani}, are known only for the contact case.  In the context of 3-manifolds, Barbot \cite{Barbot_plane_affine_geometry} proved that contact Anosov flows on 3-manifolds are necessarily $\bR$-covered and skew, while Foulon-Hasselblatt surgery \cite{FHV} produces many examples.  In fact, it is currently an open question whether every $\bR$-covered skew flow is orbit equivalent to a contact flow.   As progress towards a better understanding of these flows, we show that isomorphism of the associated contact structures is the same as orbit equivalence of flows, giving a powerful tool to use the machinery of flows to answer questions in contact geometry and vice versa. 

\begin{theorem}\label{thm_isotopy_contact}
Two contact Anosov flows on a 3-manifold are orbit equivalent if and only if their respective contact structures are contactomorphic.   They are isotopically equivalent if and only if the contact structures are isotopic.  
\end{theorem}

Recall that two contact structures $\xi_1, \xi_2$ on a manifold $M$ are {\em contactomorphic} if there exists a diffeomorphism $g\colon M \to M$ such that $g_\ast \xi_1 = \xi_2$, and they are {\em isotopic} if $g$ can be taken to be isotopic to the identity.  

We prove the reverse of Theorem \ref{thm_isotopy_contact} in Section \ref{sec:contact}, and the forward direction in the Appendix, see Theorem \ref{thm_converse_contact}.  
Using the coarse classification of tight contact structures of Colin--Giroux--Honda \cite{CGH} and the result (proved with Jonathan Bowden in the appendix) that Anosov contact structures have zero torsion, we also obtain the following.  
\begin{theorem}[Finiteness for contact Anosov flows] \label{thm:finite_contact}
On any given 3-manifold $M$, there are only finitely many contact Anosov flows on $M$ up to orbit equivalence.
\end{theorem}

Thanks to Theorem \ref{thm_isotopy_contact}, one can now fully translate results about contact Anosov flows to result about  Anosov contact structures and vice versa. We illustrate this principle in the Appendix with two examples: First, we show (Theorem \ref{thm_N_contact_structures}) that there exists hyperbolic $3$-manifolds with arbitrarily many non-contactomorphic Anosov contact structures, answering a question raised in \cite{FHV}. Second, we give a virtual description of the group of contact transformations of a  Anosov contact structure up to isotopy on some manifolds (Theorem \ref{thm_ContactMCG}). This generalizes a result by Giroux and Massot \cite{GirouxMassot}.


 \subsection{Outline of the article} 
 Section \ref{sec:main} gives a brief introduction to the structure of $\R$-covered flows on 3-manifolds, followed by the proof of Theorem \ref{thm_main}.  
 Section \ref{sec:SOE} contains the proof of Theorem \ref{thm_criterion}, and the applications to contact flows are given in Section \ref{sec:contact} and the appendix.

 \begin{acknowledgement}
 K.M.~was partially supported by NSF CAREER grant  DMS 1844516 and a Sloan Fellowship.
 T.B.~was partially supported by the NSERC (Funding reference number RGPIN-2017-04592).
 J.B.~was partially supported by the Special Priority Programme SPP 2026 Geometry at Infinity funded by the DFG.
 The authors thank Anne Vaugon and Vincent Colin for very helpful discussions. We also thank Thierry Barbot and Sergio Fenley for their detailed comments and suggestions on an earlier version of the article.
 \end{acknowledgement}
 
 
 \section{Proofs of Theorems \ref{thm:hyp_action} and \ref{thm_main}} \label{sec:main} 
The first statement of Theorem \ref{thm_main} is a special case of the second, so we in fact only need to prove that assertion.  Recall this is the statement that two flows $\varphi$ and $\psi$  are orbit equivalent if and only if there exists a homeomorphism $f\colon M \to M$ such that $f_\ast(\cP(\varphi)) = \cP(\psi)$, and if this holds the orbit equivalence can be taken to be in the isotopy class of $f$.   The forward direction is immediate; we now set-up the proof of the reverse direction.   This will lead us to the statement and proof of Theorem \ref{thm:hyp_action}, we finish the proof of Theorem \ref{thm_main} at the end of this section.  
 
Throughout the work, we assume the reader has basic familiarity with Anosov flows.  We recall below the essential structure theory of $\R$-covered flows on 3-manifolds that is  used in the proof.  Further background can be found in \cite{FH_book}, and results specific to the topological theory of Anosov flows in dimension $3$ can be found in \cite{Barbot_HDR}.
 
By work of Fenley \cite{Fen_Anosov_flow_3_manifolds} and Barbot \cite{Bar_caracterisation}, an $\R$-covered Anosov flow on a closed 3-manifold is either conjugate to the suspension of an Anosov diffeomorphism of $T^2$ or is {\em skew}, meaning that the orbit space of the lift of the flow to $\widetilde{M}$ is homeomorphic to the infinite diagonal strip 
\[\mathcal{O} = \{ (x,y) \in \mathbb{R}^2 \ | \ |x - y| <1\}\] via a homeomorphism taking the stable leaves of the flows to the horizontal cross sections of the strip, and unstable leaves to the vertical cross sections. 

Theorem \ref{thm_main} follows from purely topological considerations in the suspension case, as follows.  Suppose that $M$ is a $3$-manifold that fibers as the mapping torus of an Anosov diffeomorphism $A$ on the torus.  By a theorem of Plante \cite{Plante}, any Anosov flow on $M$ is necessarily of suspension type.  Any other suspension flow comes from a fibering of $M$ as the mapping torus of an Anosov diffeomorphism. It is a ``folklore" result that such a diffeomorphism must be conjugate to either $A$ or $A^{-1}$ --- a detailed proof can be found in \cite{Funar}.  Since our definition of orbit equivalence allows a flow to be conjugate to its inverse, we conclude that $M$ admits only one Anosov flow up to orbit equivalence. 

The case of skew flows is much more interesting.  For instance, examples of (closed, hyperbolic) manifolds that admit arbitrarily many inequivalent skew-Anosov flows were constructed in \cite{BowdenMann}.   As a first step to the proof of Theorem \ref{thm_main} in the skew case, we need to recall some general structure theory due to Barbot and Fenley that will allow us to essentially reduce the theorem to a statement about actions of $\pi_1(M)$ on $S^1$.  
 
\subsection{Skew-Anosov flows} \label{sec_skew}

Consider again the infinite diagonal strip model for the orbit space as shown in Figure \ref{fig:R-covered_case}.

\begin{figure}[h]
\begin{center}
\scalebox{0.85}{
\begin{pspicture}(-0.5,-0.5)(6,6)
\psline[linewidth=0.04cm,linestyle=dashed](2,0.5)(5,3.5)
\psline[linewidth=0.04cm,linestyle=dashed](0.5,2)(3.5,5)
\psline[linewidth=0.04cm,arrowsize=0.05cm 2.0,arrowlength=1.4,arrowinset=0.4]{->}(1,0)(5,0)
\psline[linewidth=0.04cm,arrowsize=0.05cm 2.0,arrowlength=1.4,arrowinset=0.4]{->}(0,1)(0,5)
\rput(5.2,-0.2){$\Lambda^u$}
\rput(-0.2,5.2){$\Lambda^s$}
 \psline[linewidth=0.04cm,linecolor=blue](0.5,2)(3.5,2)
 \psline[linewidth=0.04cm,linecolor=red](2,0.5)(2,3.5)
\put(1.6,1.6){$o$}
\psline[linewidth=0.04cm,linecolor=red,linestyle=dashed](3.5,2)(3.5,5)
 \rput(3.9,1.8){$o_s$}
\psline[linewidth=0.04cm,linecolor=blue,linestyle=dashed](2,3.5)(5,3.5)
\rput(1.7,3.7){$o_u$}
\put(3.6,3.6){$\eta\left(o\right)$ }
\end{pspicture}
}
\end{center}
 \caption{The orbit space $\orb$} \label{fig:R-covered_case}
\end{figure}
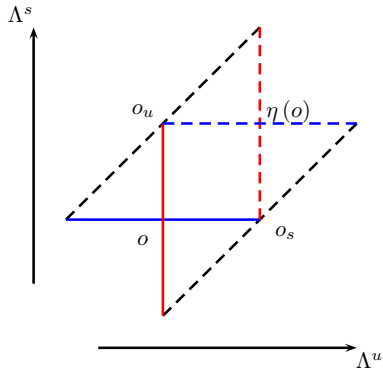

In this model, each point $o \in \mathcal{O}$ can be assigned a point $o_u$ on the upper boundary of the strip by following the unstable leaf through $o$, and a point $o_l$ on the lower boundary by following the stable leaf. Taking the intersection of the stable leaf through $o_u$ and unstable through $o_l$ defines a continuous, fixed point free map $\eta\colon \mathcal{O} \to \mathcal{O}$, that we call the {\em half-step up} map.\footnote{this map and its square have both been referred to as the one-step-up map elsewhere in the literature.  We choose to call the square the one-step-up and $\eta$ the half-step-up}  
This map exchanges stable leaves and unstable leaves, so $\tau = \eta^2$ induces a map on the leaf space $\Lambda^s$ of the weak stable foliation. We call $\tau$ the \emph{one step up map}.

If the weak foliations are transversely orientable, then $\tau$ commutes with the action of $\pi_1(M)$ on $\Lambda^s$.  Identifying $\Lambda^s \cong \R$ so that $\tau$ is identified with the translation $x \mapsto x+1$ realizes $\pi_1(M)$ as a subgroup of $\HomeoZ$, the group of orientation-preserving homeomorphisms of $\R$ commuting with integer translations.
In the non-orientable case, $\tau$ is twisted-$\pi_1$-equivariant: for any $\gamma \in \pi_1(M)$ that reverses the orientation of $\Lambda^s$ we have $\tau(\gamma x) = \gamma \tau^{-1}(x)$. 

Our perspective going forward will be to study the flows through the action of $\pi_1(M)$ on the leaf space.  
For simplicity, we assume first that the weak foliations are transversely orientable, and then state the necessary modifications for the non-orientable case in Section \ref{sec:nonorientable}.   In the orientable case, the dynamics of the action of $\pi_1(M)$ on $\Lambda^s$ is what Thurston \cite{slithering} calls an {\em extended convergence group action}.  However, the only dynamical property that we will need is the fact that in such a group, any orientation-preserving homeomorphism with fixed points has exactly two fixed points in $[0,1)$, one attracting and one repelling.  This can be seen directly from the Anosov dynamics of the flow.   

Another dynamical property that will be of use comes from Barbot \cite[Theorem 2.5]{Bar_caracterisation}.  Barbot's theorem states that skew-Anosov flows are transitive and the action of the group generated by $\pi_1(M)$ and $\tau$ on $\R$ is {\em minimal}, meaning that all orbits are dense.   Since $\R/\tau \cong S^1$ and the action of $\pi_1(M)$ descends to this circle, this latter statement is equivalent to the statement that the action of $\pi_1(M)$ on $\R/\tau$ is minimal.  The reader may find it useful to visualize the action on $\R$ by thinking of it on the circle $\R/\tau$, but for simplicity we will work in  $\HomeoZ$. 
The next subsection establishes some general results about such subgroups of $\HomeoZ$, concluding with the proof of Theorem \ref{thm:hyp_action}.    

\subsection{Hyperbolic-like homeomorphisms: Proof of Theorem \ref{thm:hyp_action}}

\begin{definition}
We say an element $f \in \HomeoZ$ is {\em hyperbolic-like} if it has exactly two fixed points in $[0,1)$, one attracting and one repelling, and a group action $\rho: G \to \Homeo(\bR)$ is hyperbolic-like if its image lies in $\HomeoZ$ and every element with fixed points is hyperbolic-like.  
\end{definition} 

In the context of the action of the fundamental group of a 3-manifold with a skew-Anosov flow on the stable leaf space (which is what we have in mind), Thurston \cite{slithering} calls hyperbolic-like elements {\em space-like} homeomorphisms.   Since this notation is not commonplace, we have chosen the terminology ``hyperbolic-like" since the induced action of such elements on $\bR/\bZ$ are topologically conjugate to hyperbolic M\"obius transformations.  

If $a$ is a hyperbolic-like element, we will use the notation $a_+$ and $a_-$ to denote an attracting, respectively, repelling, fixed point for $a$.  
For two hyperbolic-like elements $a$ and $b$, we say the fixed sets of $a$ and $b$ are {\em linked} if each connected component of $\R \smallsetminus \fix(a)$ contains a fixed point for $b$, or if $a$ and $b$ have fixed points in common.  We say they are {\em unlinked} otherwise.   By convention, when we speak about the order of fixed points on $\bR$, we use the notation 
\[ a_+ < b_- < a_- < b_+ \]
to mean that there exist four {\em consecutive} elements of $\fix(a) \cup \fix(b)$, ordered as indicated by the inequality.  In this example, the fixed sets of $a$ and $b$ are linked.   Since both $a$ and $b$ commute with integer translation, the next element of $\fix(a) \cup \fix(b)$ to the right of $b_+$ is another attracting fixed point for $a$, equal to $a_+ + 1$.

The next series of lemmas shows that the configuration of fixed points of a pair or of a triple of elements can be detected by the set of words in those elements which act with fixed points.

\begin{lemma} \label{lem:linked}
Let $a$ and $b$ be hyperbolic-like elements of $\HomeoZ$.  The fixed sets of $a$ and $b$ are linked if and only if $\fix(a^n b^m) \neq \emptyset$ for all $n, m \in \bZ$.  
\end{lemma}

\begin{proof} Suppose first that the fixed sets of $a$ and $b$ are linked, and let $n, m$ be given.  Since the property of having linked fixed sets does not change after passing to inverses, up to replacing $a$ or $b$ with their inverses we can assume that $n, m \geq 0$.  Since the fixed sets of $a$ and $b$ are linked, either $a$ and $b$ have a common fixed point (in which case we are done) or there exists some connected component of $\R \smallsetminus (\fix(a) \cup \fix(b))$ bounded on one side by an attracting fixed point for $a$, and on the other by an attracting fixed point for $b$.  Let $I$ denote the closure of this component.  Then $a^n b^m(I) \subset I$, so $a^n b^m$ has a fixed point in $I$.  

To prove the converse, suppose now that the sets are unlinked.  Up to passing to inverses, we may find consecutive attracting and repelling fixed points for $a$ and $b$ that lie in the order 
\[ a_+ < a_- < b_+ < b_- \] 
with no other fixed points between $a_+$ and $b_-$.  For $m$ large enough, $b^m(a_+)$ will lie in the open interval $(a_-, b_+)$.  For $n$ large enough, $a^nb^m(a_+)$ will therefore lie in the open interval $(b_-, a_+ + 1)$.  Similarly,  $b^ma^nb^m(a_+)$ will lie to the right of $a_+ + 1$, as will $a^nb^ma^nb^m(a_+)$.  Thus, $(a^n b^m)^2$ translates some point a distance at least 1.  Any such element of $\HomeoZ$ is fixed point free, hence its root $a^n b^m$ is fixed point free as well.  
\end{proof} 

The next lemma says that if $a$ and $b$ have unlinked fixed sets, then we can detect the cyclic order of attracting and repelling fixed points by understanding which words in $a$ and $b$ have fixed points.  
\begin{lemma} \label{lem:unlinked}
Suppose $a$ and $b$ have unlinked fixed sets. The word $b^Na^N$ has a fixed point for every $N>0$ if and only if one may find a set of consecutive fixed points either in the order
\[ a_+ < a_- < b_- < b_+ \] 
or that obtained from the above by replacing $a$ and $b$ simultaneously with their inverses.
\end{lemma} 

\begin{proof} 
If the ordering shown above occurs, then the interval $[b_+, a_+ +1]$ is mapped into itself by $b^Na^N$, so the map has a fixed point.  
Note that the ordering obtained by replacing $a$ and $b$ with their inverses can also be obtained simply by reversing the orientation of $\bR$ and considering a sequence of consecutive fixed points starting with $a_-$.  Since reversing orientation of $\bR$ obviously does not change the property of a map having fixed points, we have already proved one direction of the lemma.   
For the converse, if the ordering above does not occur even after passing to inverses, then we have either the order
\[ a_+ < a_- < b_+ < b_- \] 
or 
\[ a_- < a_+ < b_- < b_+ .\] 
In the first case, for sufficiently large $N$ we have $(b^Na^N)^2(b_+) > b_+ + 1$ so the map is fixed point free, and in the second case we have $(b^Na^N)^2(b_+) < b_+ - 1$ so the map is again fixed point free.  
\end{proof}

Our next goal is to use this information to reconstruct a minimal action, up to conjugacy, from the data of the ordering of fixed point sets.   For this we need an elementary lemma. 

\begin{lemma} \label{lem:dense}
Suppose $G \subset \HomeoZ$ is a nonabelian group whose action on $\bR$ is minimal and hyperbolic-like.  Then for any point $x \in \R$ and any $\epsilon > 0$, there exists $a \in G$ such that $a$ has two fixed points in the $\epsilon$-neighborhood of $x$.   
\end{lemma} 

\begin{proof} 
If no nontrivial element of $G$ acts with fixed points, then $G$ would be abelian, by H\"older's theorem, so by assumption this case does not occur.   Thus, there exist hyperbolic-like elements, and by minimality of the action of $G$, the set of their attracting fixed points is dense.

Let $x$ and $\epsilon >0$ be given.    
Fix any hyperbolic-like element $g$ with an attracting fixed point in the $\epsilon/2$ neighborhood of $x$. 
Observe that, if $f$ and $g$ are hyperbolic-like, and $f$ does not fix a repelling point $g_-$ for $g$, then all fixed points of the conjugate $g^N f g^{-N}$ approach the attracting fixed points of $g$ as $N \to \infty$.   Thus, it suffices to find a hyperbolic-like $f$ that does not fix $g_-$.   By minimality, there exists $h \in G$ such that $h(g_-)$ lies strictly between $g_-$ and $g_+$, where $g_- < g_+$ are consecutive fixed points.   If $h(g_+ - 1) \neq g_-$, then $hgh^{-1}$ is hyperbolic-like and has fixed points distinct from $g_-$, since they are the images of the fixed points of $g$ under $h$ and we are done.   If instead $h(g_+ -1) = g_-$, or equivalently $h(g_+) = g_-+1$, then we have
\[g_- <  h^{-1}(g_+) < g_+  < g_- +1\] 
and therefore $h^{-1}gh$ has fixed points $g_+$ and $h^{-1}(g_+) \neq g_-$.
\end{proof} 

Although not strictly needed in our proof, Lemma \ref{lem:dense} can be strengthened to the following density for pairs of fixed points.  
\begin{lemma} \label{lem:pairs_dense}
Suppose $G \subset \HomeoZ$ is a nonabelian group whose action on $\bR$ is minimal and hyperbolic-like.  
Given $x, y \in [0,1)$, and $\epsilon >0$, there exists a hyperbolic-like element $g \in G$ with fixed points satisfying $|g_- - x| < \epsilon$ and $|g_+ - y| < \epsilon$. 
\end{lemma} 

\begin{proof} 
Let $x, y \in [0,1)$, and $\epsilon >0$ be given.  Without loss of generality, assume $x < y$ and assume that $\epsilon$ is small enough so that the $\epsilon$-neighborhoods of $x$ and $y$ and all of their integer translates are pairwise disjoint.  
By lemma \ref{lem:dense} we may find hyperbolic-like elements $a$ and $b$ with fixed points in the $\epsilon/2$-neighborhoods of $x$ and $y$ respectively.   
Replacing $a$ or $b$ with their inverses if needed, we can assume these fixed points are ordered
\[ a_+ < a_- < b_- < b_+ \] 
By Lemma \ref{lem:unlinked}, this implies that $b^Na^N$ has a fixed point for every $N>0$.  Furthermore, if $N$ is sufficiently large, an attracting fixed point for $b^Na^N$ will lie within the $\epsilon/2$ neighborhood of $b_+$, and a repelling fixed point within the $\epsilon/2$ neighborhood of $a_-$, this is simply because $b^Na^N$ takes a complement of the $\epsilon/2$-neighborhood of the union of translates of $a_-$ to a neighborhood of the attracting fixed points for $b$.   
\end{proof}

We can now finish the proof of our first main theorem; first we recall the statement.  

\hypthm*

\begin{proof} 
Let $G$ be nonabelian, hyperbolic-like and acting minimally on $\bR$, and $\rho(G)$ another such faithful action of $G$ on $\bR$, with the same set of hyperbolic-like elements as $G$.  Let $g \in G$ be a hyperbolic-like element (recall that such an element exists by H\"older's theorem since $G$ is nonabelian).  Choose coordinates on $\R$ so that the attracting fixed points of $g$ and of $\rho(g)$ are precisely the integers.  
We need to choose an orientation on the line $\bR$ on which $\rho$ acts. To do this take some $f \in G$ with an axis unlinked with $g$.  Such a map $f$ exists by Lemma \ref{lem:dense}.  Replacing $f$ with its inverse if needed, we can find such a map so that consecutive fixed points are ordered
\[ g_+ < g_- < f_- < f_+ .\] 
By Lemma \ref{lem:linked} and the fact the action of $\rho(G)$ has the same elements with fixed points as the original action of $G$, we conclude that the fixed sets of $\rho(g)$ and $\rho(f)$ are unlinked.  By the same reasoning, using Lemma \ref{lem:unlinked}, the map $f^Ng^N$ has a fixed point for all $N>0$, so by hypothesis the same is true for $\rho(f)^N \rho(g)^N$.  Applying the other direction of Lemma \ref{lem:unlinked} we can now fix an orientation on $\R$ so that consecutive fixed points of $\rho(f)$ and $\rho(g)$ are ordered 
\[ \rho(g)_+ < \rho(g)_- < \rho(f)_- < \rho(f)_+ .\] 
Our next goal is to show that this determines the ordering of the set of all attracting fixed points of hyperbolic-like elements of $\rho(G)$.  We then define a map $\Theta$ on the (dense) subset of $\R$ consisting of attracting fixed points of other elements by sending the unique attracting fixed point of an element $h$ that lies in $[m, m+1)$ to the unique attracting fixed point of $\rho(h)$ in $[m, m+1)$.   
This order-preserving property is sufficient to show that our map $\Theta$ is continuous, from which it will easily follow that it can be extended continuously to a homeomorphism of $\R$ that conjugates the action of $G$ and $\rho(G)$.    

Consider first an element $h$ with fixed set that is unlinked with both $f$ and $g$. Up to switching $h$ with $h^{-1}$, there are three cases to consider.\\

\noindent\textbf{Case 1.} Suppose first that we have the ordering
\[ g_+ < g_- < f_- < h_- < h_+  < f_+ .\] 
Applying Lemma \ref{lem:unlinked}, we have that $f^Nh^N$ and $g^Nh^N$ have fixed points for all positive $N$, hence so does $\rho(f)^N \rho(h)^N$ and $\rho(g)^N \rho(h)^N$. Applying the lemma again implies that one of three orderings occur, either
\[ \tag{$\ast$} \label{eq_star}  \rho(g)_+  < \rho(h)_+ < \rho(h)_-  < \rho(g)_- < \rho(f)_- < \rho(f)_+, \]
\[ \tag{$\ast \ast$} \label{eq_starstar}  \rho(h)_+  < \rho(g)_+ < \rho(g)_-  < \rho(h)_- < \rho(f)_- < \rho(f)_+, \]
or 
\[ \tag{$\ast \ast \ast$} \label{eq_starstarstar} \rho(g)_+ < \rho(g)_- < \rho(f)_-  < \rho(h)_- < \rho(h)_+  < \rho(f)_+ .\]
 We want to show that only case \eqref{eq_starstarstar} can occur.   We will show that  \eqref{eq_star} does not occur.  Eliminating the possibility of \eqref{eq_starstar} is done by exactly the same argument, swtiching the roles of $g$ and $h$; we omit the details.  

Since $\fix(aba^{-1}) = a \fix(b)$, for any $n>0$, we also have 
\[ g_+ < g_- < f_- < (f^n h f^{-n})_- < (f^n h f^{-n})_+  < f_+\]
and so $g^N (f^nhf^{-n})^N$ has a fixed point for all $N >0$, as does its image under $\rho$.    
If ordering \eqref{eq_star} were to occur, then for sufficiently large $n$ we would have 
\[  \rho(g)_+  < \rho(g)_- < \rho(f)_- < \rho(f)_+  < \rho(f^n h f^{-n})_+ < \rho(f^n h f^{-n})_-  \] 
contradicting Lemma \ref{lem:unlinked} applied to $\rho(g)^N \rho(f^nhf^{-n})^N$.  Thus, the ordering of consecutive fixed points of $f, g$ and $h$ under $\rho$ agrees with that for the original action. 
\smallskip

\noindent \textbf{Case 2.} The ordering
\[ g_+ < h_- < h_+ < g_- < f_- < f_+ .\] 
is handled exactly as above, exchanging the roles of $f$ and $g$.
\smallskip

\noindent \textbf{Case 3.} 
Now, suppose instead that the ordering of the fixed points of $f$, $g$ and $h$ is
\[
  g_+ < g_- < h_+ < h_- < f_- < f_+ .
\]
Consider the elements $a= g^nf^n$ and $b= f^nh^n$ for some large positive $n$. As $n \to \infty$ the attracting fixed point $a_+$ approaches $f_+$, and similarly $a_-$ approaches $g_-$,  $b_+$ approaches $h_+$, and $b_-$ approaches $f_-$.  Thus, provided $n$ is chosen large enough, we have 
\[
 g_+ < g_- < a_- <b_+ < h_+ < h_-< b_- < f_- < f_+ < a_+. 
\]
We can then apply the previous cases to the triples $(g,a,h)$, $(g,a,f)$, $(g,a,b)$, and $(g,b,h)$ to show the ordering of their fixed points is each preserved by $\rho$.  We deduce that the ordering of the fixed points of $\rho(f)$, $\rho(g)$ and $\rho(h)$ matches that of the fixed points of $f$, $g$, and $h$.

We can now quickly finish the proof of the Theorem.  
Suppose we have some hyperbolic-like $a$ and $b$ with 
\[0 \leq a_+ < b_+ < 1.\]

Rather than consider cases depending on whether $a$ and $b$ are linked or not, we can instead use Lemma \ref{lem:dense} to choose $c$ and $d$ in $G$ with fixed sets very close to $a_+$ and $b_+$ such that $c$ and $d$ have fixed sets unlinked with $a$ and $b$. 
Suppose for simplicity that $0 < a_+$ (otherwise, simply replace 0 in what follows with some very small $-\epsilon$ and $1$ with $1-\epsilon$, and repeat the proof). By Lemma \ref{lem:pairs_dense}, we can choose elements $c$ and $d$ with fixed points so that we have the ordering
\[ 
0 < c_- < c_+ < a_+ <  d_- < d_+ < b_+ < 1
\]
and additionally have that $a$ does not have a repelling fixed point between $c_-$ and $a_+$, and  $b$ does not have a repelling fixed point between $d_-$ and $d_+$.   We can also choose such $c$ and $d$ so that fixed sets are each unlinked with respect to $f$ and to $g$, so that we may apply the observation above to $c$ and $d$ and determine the relative order of their fixed point sets.  Any choice of $c$ and $d$ with fixed sets sufficiently close to $a_+$ and $b_+$ will have this property.   Thus, by our convention on $\rho(g)$, we conclude that 
\[\rho(g)_+ = 0 < \rho(c)_- < \rho(c)_+ < \rho(d)_- < \rho(d)_+  < 1.\]
Since $c$ and $d$ had unlinked fixed points with respect to $a$ and $b$ and $g$, we can apply the observation again and conclude that the ordering of fixed sets is preserved, namely 
\[\rho(g)_+ = 0 < \rho(c)_- < \rho(c)_+ < \rho(a)_+ < \rho(d)_- < \rho(d)_+ < \rho(b)_+ < 1\]
and in particular, $\rho(g)_+ = 0 < \rho(a)_+ < \rho(b)_+ < 1$ as we needed to show.  

Thus, we have defined an order-preserving (and hence continuous) injective map between two dense subsets of $\R$. This extends uniquely to a continuous map $\R \to \R$ with continuous inverse, which we denote by $\Theta$.  
It remains to see that $\Theta$ conjugates the actions of $G$ and $\rho(G)$.  
Let $g \in G$ be given.  Note that, if $a_+$ is a fixed point of a hyperbolic-like element of $a \in G$, then $g(a_+)$ is an attracting fixed point of $(gag^{-1})$, thus $\Theta(ga_+)$ is some attracting fixed point of $\rho(gag^{-1})$, i.e.~the image of an attracting fixed point of $\rho(a)$ under $\rho(g)$.  In other words, $\Theta(ga_+) = \rho(g) \Theta(a_+) + n$ for some $n \in \mathbb{Z}$.  Since $\Theta$ is continuous and fixed points of hyperbolic-like elements are dense in the source and the target, we conclude that $n=0$, and $g$-equivariance holds on a dense set, hence everywhere.
\end{proof} 

\begin{rem}
If $G$ is not assumed to act minimally (but $\rho(G)$ is), the same proof strategy can be used to produce a {\em semi-conjugacy} between the actions of $G$ and $\rho(G)$, defined on the closure of the $G$-invariant set consisting of hyperbolic fixed points.  
\end{rem}

\subsection{Conclusion of the proof of Theorem \ref{thm_main}} 
Returning to the set-up of Theorem \ref{thm_main}, suppose that $\varphi$ and $\psi$ are two skew-Anosov flows with $f_\ast \cP(\varphi) = \cP(\psi)$, for some homeomorphism $f\colon M \to M$.  Replacing $\varphi$ with its conjugate under $f$, we obtain a flow $\varphi$ satisfying $\cP(\varphi) = \cP(\psi)$.   What we need to show is that $\varphi$ and $\psi$ are orbit equivalent by some homeomorphism of $M$ that is isotopic to the identity.  

\subsubsection*{Orientable case}
Assume first that $\Lambda^s(\varphi)$ and  $\Lambda^s(\psi)$ are both transversely orientable. 
We then will apply Theorem \ref{thm:hyp_action} to the group $G := \pi_1(M)$. Recall the action of this group on the leaf space $\Lambda^s(\varphi) \cong \R$ is faithful, minimal and has the property that all elements with fixed points are hyperbolic-like, under the parametrization of $\R$ where $\tau(\varphi)$ acts as translation by 1, and the same holds for $\psi$.

That $\cP(\varphi) = \cP(\psi)$ means precisely that these two representations have the same elements with fixed points. 
Thus, Theorem \ref{thm:hyp_action} implies that these two actions are conjugate.  
This also gives a conjugacy between the actions on the unstable leaf spaces $\Lambda^u(\varphi)$ and $\Lambda^u(\psi)$ via further conjugation by the half step up map $\eta$.  Considering intersections of stable and unstable leaves, we can promote this to a conjugacy of the actions of $\pi_1(M)$ on the orbit space $\mathcal{O}$. Following an argument of Ghys using Haefliger's theory of classifying spaces of foliations, Barbot showed using an averaging trick that such a conjugacy on the orbit space can always be realized by a homeomorphism of $M$ giving an orbit-equivalence of the flows \cite[Theorem 3.4]{Bar_caracterisation}. 
In our case, this homeomorphism  is easily seen to be isotopic to the identity by considering the action on $\pi_1$.
This concludes the proof in the transversely orientable case.

\subsubsection*{General case.} \label{sec:nonorientable}
For the general case, consider again the action of $\pi_1(M)$ on $\Lambda^s(\varphi) \cong \R$ and on $\Lambda^s(\psi)$.  Let $G \subset \pi_1(M)$ be the normal subgroup generated by all squares of elements. Since each element of $G$ is a product of squares, its action on $\Lambda^s(\varphi)$ and  $\Lambda^s(\psi)$ is by orientation preserving homeomorphisms.   The proof of \cite[Th\'eor\`eme 2.5]{Bar_caracterisation} shows directly that the action of $G$ on $\Lambda^s(\varphi)$ and $\Lambda^s(\psi)$ is also minimal.
Thus, we may apply Theorem \ref{thm:hyp_action} and conclude that the actions of $G$ on the respective leaf spaces are conjugate.  We wish to show that this conjugacy extends to a conjugacy of the actions of $\pi_1(M)$.  

Apply a conjugacy so that the actions of $G$ on $\Lambda^s(\varphi)$ and $\Lambda^s(\psi)$ agree.  Now our goal is to show that, after possibly further conjugating by an integer translation (which commutes with the action of $G$), the actions of $\pi_1(M)$ agree.   
Let $\rho_\varphi$ and $\rho_\psi$ denote the actions on $\Lambda^s(\varphi)$ and $\Lambda^s(\psi)$ respectively, assumed to agree on the restrictions to $G$.  

Note that if $\gamma \in \pi_1(M)$, and $x \in \R$ is an attracting fixed point of $\rho_\varphi(g)$ for some element $g \in G$, then $\rho_\varphi(\gamma)(x)$ is an attracting fixed point of $\rho_\varphi(\gamma g \gamma^{-1})$, where we have $\gamma g \gamma^{-1} \in G$.  The same applies to $\rho_\psi(\gamma)$.  The set of all attracting fixed points for elements of $G$ is dense, and each element with fixed points has a $\bZ$-invariant set of attracting fixed points with exactly one in $[0,1)$.    

First we verify that  $\rho_ \varphi(\gamma)$ preserves orientation if and only if $\rho_\psi(\gamma)$ does.  Suppose $\rho_ \varphi(\gamma)$ reverses orientation.  Let $g_1, g_2$, and $g_3$ in $G$ be elements with attracting fixed points satisfying 
\[ \rho_\varphi(g_1)_+ < \rho_\varphi(g_2)_+ < \rho_\varphi(g_3)_+ < \rho_\varphi(g_1)_+ + 1 \]
Since $\rho_ \varphi(\gamma)$ reverses orientation, we have 
\[ \rho_\varphi(\gamma g_1 \gamma^{-1})_+ > \rho_\varphi(\gamma g_2 \gamma^{-1})_+ > \rho_\varphi(\gamma g_3 \gamma^{-1})_+ > \rho_\varphi(\gamma g_1 \gamma^{-1})_+ - 1 \]
whereas if $\rho_\varphi(\gamma)$ preserved the orientation, the original order would not be affected by conjugation.   
Since $\gamma g \gamma^{-1} \in G$, this ordering holds also for the action under $\rho_\psi$, showing that $\rho_\psi(\gamma)$ necessarily reverses orientation as well.  The situation being symmetric, we have proved the claimed if and only if statement.  

Now, consider the subgroup $P$ of elements of $\pi_1(M)$ whose action preserves orientation on $\Lambda^s(\varphi)$; or, as we have just shown, equivalently preserves orientation of $\Lambda^s(\psi)$.  
Our description above also implies that for any $\gamma\in P$, we have $\rho_\varphi(\gamma) = \rho_\psi(\gamma) \circ T_\gamma$ for some integer translation $T_\gamma$. Since orientation-preserving elements commute with integer translation, the map $\gamma \in P \mapsto T_\gamma$ is a group homomorphism. Now, when $\gamma \in G$, $T_{\gamma}$ is the identity, thus $T_{\gamma}$ is the identity for any $\gamma \in P$. So in particular, the actions of $\rho_\varphi$ and $\rho_\psi$ are identical on $P$.

If instead we consider $\gamma$ an element reversing the orientation on the leaf spaces, then $\rho_\psi(\gamma)$ and $\rho_\varphi(\gamma)$ each have a unique fixed point. Their action being determined modulo integer translations means that 
there exists an integer translation $T'_\gamma$ such that $\rho_\varphi(\gamma) = {T'}_\gamma \rho_\psi(\gamma) {T'}_\gamma^{-1}$.
Fix some orientation-reversing element $\gamma_0$.  Up to conjugating the action of $\rho_\psi$ by an integer translation, we can assume that $T'_{\gamma_0}$ is the identity, i.e.~ that $\rho_\psi(\gamma_0) = \rho_\varphi(\gamma_0)$.    
Notice that this conjugation does not affect the fact that $\rho_\varphi$ and $\rho_\psi$ are identical on $P$, since the action of elements in $P$ commutes with integer translations.   We wish to show now that $T'_\gamma = 0$ for all $\gamma \in \pi_1(M) \smallsetminus P$, so the actions agree. 

Let $\gamma$ be any orientation reversing element.  Then $\gamma_0\gamma$ preserves the orientation so $\rho_\psi(\gamma_0 \gamma) = \rho_\varphi(\gamma_0 \gamma)$. As $\rho_\psi(\gamma_0) = \rho_\varphi(\gamma_0)$, we deduce directly that $\rho_\psi(\gamma) = \rho_\varphi(\gamma)$.
Hence, we proved that the actions $\rho_\varphi$ and $\rho_\psi$ are the same, ending the proof of Theorem \ref{thm_main}.

\section{Classifying self orbit-equivalences} \label{sec:SOE}

This section gives the proof of Theorem \ref{thm_criterion}.  We start by introducing some additional necessary background material on the orbit space of the flow.

Returning to the picture from Section \ref{sec_skew}, recall that the orbit space of a skew Anosov flow is homeomorphic to a diagonal strip in $\bR^2$  foliated by $\Lambda^s$ and $\Lambda^u$ in the two coordinate directions.  For each orbit $o\in \orb$, the ideal quadrilateral in $\orb$ with corners $o$ and $\eta (o)$ and sides the stable and unstable (half-)leaves of $o$ and $\eta (o)$ is called a \emph{lozenge}.  
The union of lozenges associated with the orbits $\eta^{n}(o)$, $n\in \bZ$ is called a \emph{string of lozenges}.  The reader may consult~\cite[section 2]{BF_counting} for more background about lozenges in general Anosov flows.

\begin{figure}[h]
\begin{center}
\scalebox{0.85}{
\begin{pspicture}(-0.5,-0.5)(6,6)
\psline[linewidth=0.04cm,linestyle=dashed](1.5,0.5)(4.5,3.5)
\psline[linewidth=0.04cm,linestyle=dashed](0.5,1.5)(3.5,4.5)
\psline[linewidth=0.04cm,arrowsize=0.05cm 2.0,arrowlength=1.4,arrowinset=0.4]{->}(1,0)(5,0)
\psline[linewidth=0.04cm,arrowsize=0.05cm 2.0,arrowlength=1.4,arrowinset=0.4]{->}(0,1)(0,5)
\rput(5.2,-0.2){$\Lambda^u$}
\rput(-0.2,5.2){$\Lambda^s$}
 \psline[linewidth=0.04cm,linecolor=blue](0.5,1.5)(2.5,1.5)
 \psline[linewidth=0.04cm,linecolor=red](1.5,0.5)(1.5,2.5)
 \psline[linewidth=0.04cm,linecolor=blue](1.5,2.5)(3.5,2.5)
 \psline[linewidth=0.04cm,linecolor=red](2.5,1.5)(2.5,3.5)
 \psline[linewidth=0.04cm,linecolor=blue](2.5,3.5)(4.5,3.5)
 \psline[linewidth=0.04cm,linecolor=red](3.5,2.5)(3.5,4.5)
\end{pspicture}
}
\end{center}
 \caption{The orbit space $\orb$ with (part of) a string of lozenges} \label{fig:orbit_space}
\end{figure}
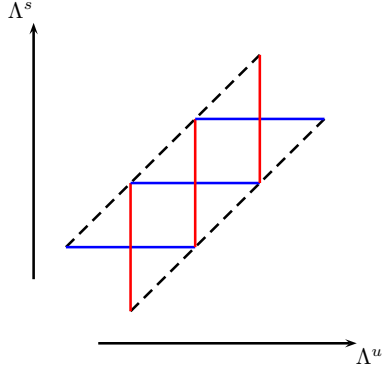

Recall (\cite{Bar_MPOT}) that an (immersed) incompressible torus $T$ is {\em quasi-transverse} to an Anosov flow $\varphi$ if it is transverse everywhere except along finitely many periodic orbits of $\varphi$.
Barbot \cite[Th\'eor\`eme C]{Bar_MPOT} (see also \cite[Theorem 6.10]{BarbotFenley_pA_toroidal}) showed that any incompressible embedded torus in $M$ can be isotoped to a quasi-transverse torus, except if it is the boundary of a tubular neighborhood of an embedded one-sided Klein bottle.  Such embeddings of a Klein bottle do not actually arise if the flow is transversally orientable and $\bR$-covered. We were not able to locate a proof of this fact in the literature so add it as a lemma below. We thank Sergio Fenley for discussing this fact with us.

\begin{lemma} \label{lem:klein}
 Let $\varphi$ be a transversally orientable skew $\R$-covered Anosov flow on $M$. Then $M$ does not contain a $\pi_1$-injective immersion of the Klein bottle.
\end{lemma}

\begin{proof}
 Suppose there is a subgroup $G$ of $\pi_1(M)$ isomorphic to the fundamental group of the Klein bottle, i.e., we have non trivial elements $a,b\in \pi_1(M)$ such that $a b a^{-1} = b^{-1}$ and $G = \langle a,b\rangle$.  
Since $\varphi$ is transversely orientable, $G$ acts faithfully on $\Lambda^s(\varphi) \simeq \bR$ by orientation preserving homeomorphisms commuting with integer translations.   Consider the induced action of $G$ on $\bR/ \eta^2$.  For this induced action, both $a$ and $b$ have a single attracting and single repelling fixed point.  Since $a b a^{-1} = b^{-1}$, we must have that $a$ interchanges the attracting and repelling points for $b$.  But this gives $a$ an orbit of period two, contradicting that it acts with fixed points. 
\end{proof}

The main result that we need is the following: 

\begin{lemma}[Barbot \cite{Bar_MPOT}] \label{lem:lozenges}
 If $T$ is a quasi-transverse torus in $M$ for a skew Anosov flow, then each lift $\wt T$ of $T$ to the universal cover projects to a unique string of lozenges $\cC$ in $\orb$. That string of lozenges is the unique string of lozenges left invariant by the $\bZ^2$ subgroup of $\pi_1(M)$ that fixes $\wt T$.  The interior of each lozenge in the string corresponds to a maximal annulus in $T$ transverse to the flow; its corners are the periodic orbits bounding the annulus.   
  
Conversely, if $G$ is a subgroup of $\pi_1(M)$ isomorphic to  $\bZ^2$, then there exists a unique string of lozenges in $\orb$ fixed by $G$. Moreover that string of lozenge is the projection of an (a priori only immersed) torus $T$.
\end{lemma}
\begin{rem}
We have stated this lemma in the special case of skew Anosov flows. For general Anosov flows on $3$-manifold, one needs to replace ``string of lozenges" by ``chain of lozenges'', see \cite[section 2]{BF_counting}, and some $\bZ^2$ subgroups may fix two distinct chains of lozenges.
\end{rem}

We also need the following easy proposition giving another condition for an element $\gamma \in \pi_1(M)$ to be represented by a periodic orbit of a skew Anosov flow $\varphi$.
Recall that $\Lambda^s \cong \bR$ denotes the leaf space of the weak stable foliation of the lifted flow $\wt \varphi$. We denote the natural projection by $\pi^s \colon \wt M \to \Lambda^s$.

\begin{proposition}\label{prop_charac_periodic_orbits}
 Let $\varphi$ be a skew Anosov flow on a $3$-manifold $M$.
 Let $\gamma \in \pi_1(M)$. The following are equivalent:
 \begin{enumerate}[label=(\roman*)]
  \item \label{item_represented} The element $\gamma$ is represented by a periodic orbit of $\varphi$;
  \item \label{item_onebounded} There exists a $\gamma$-invariant curve $c \subset \wt M$ such that $\pi^s(c)$ is bounded (either above or below) in $\Lambda^s$;
  \item \label{item_allbounded} For any $\gamma$-invariant curve $c \subset \wt M$, its image $\pi^s(c)$ is bounded in $\Lambda^s$.
 \end{enumerate}
\end{proposition}

\begin{proof} 
Item \ref{item_allbounded} trivially implies \ref{item_onebounded}, we will show that item \ref{item_onebounded} implies \ref{item_represented} and that \ref{item_represented} implies \ref{item_allbounded}.

\ref{item_onebounded} $\Rightarrow$ \ref{item_represented}:  Assume that there exists a $\gamma$-invariant curve $c \subset \wt M$ such that $\pi^s(c)$ is bounded above in $\Lambda^s$, and let $L^+$ denote its upper bound.  (The case where $\pi^s(c)$ is bounded below is analogous.) 
If the action of $\gamma$ reverses the orientation of $\Lambda^s$, then it fixes a leaf which necessarily contains a periodic orbit represented by $\gamma$, and we are done.  
If $\gamma$ instead preserves the orientation, then the upper bound $L^+ \in \Lambda^s$ is a $\gamma$-invariant point since  $\pi^s(c)$ is $\gamma$--invariant, and thus the leaf corresponding to this point contains a unique $\gamma$-invariant orbit. 
This shows item \ref{item_represented}.  
 
\ref{item_represented} $\Rightarrow$ \ref{item_allbounded}:  Suppose $\gamma$ is represented by a periodic orbit of the flow and $c \subset \wt M$ is a $\gamma$-invariant curve. 
Supposing first that $\gamma$ preserves orientation on $\R$, the action has an attracting fixed point $x \in \Lambda^s$. Moreover, the images of this point under powers of the half-step up map are alternating attracting and repelling fixed points of $\gamma$; if $i$ is even, then $\eta^i(x)$ is an attracting fixed point, if $i$ is an odd integer then $\eta^i(x)$ is a repelling fixed point, as described in Section \ref{sec_skew}.   In the case where $\gamma$ reverses orientation, then the same argument above applies if we replace $\gamma$ with $\gamma^2$, which preserves orientation. 

In either case, since the action of $\gamma^2$ on $c$ is free, we may take a compact fundamental domain $I$ for the action.  The projection $\pi^s(I)$ is contained in some bounded interval $[\eta^{-k}(x), \eta^k(x)]$. Therefore, $\pi^s(c) = \cup_{n\in \bZ} \gamma^{2n} \cdot \pi^s(I)\subset [\eta^{-k}(x), \eta^k(x)]$, so it is bounded, proving item \ref{item_allbounded}.
\end{proof}

Combining Theorem \ref{thm_main} with Proposition \ref{prop_charac_periodic_orbits} gives the following as a direct consequence. 

\begin{theorem} \label{thm:bounded}
 Let $\varphi$ and $\psi$ be two skew $\bR$-covered Anosov flows on a $3$-manifold $M$.
 The flows $\varphi$ and $\psi$ are isotopically equivalent if and only if for any periodic orbit $\alpha$ of $\psi$ (resp.~$\varphi$) with lift $\wt\alpha \subset \wt M$, the projection $\pi_1^s(\wt\alpha) \subset \Lambda^s(\varphi)$ (resp.~$\pi_1^s(\wt\alpha) \subset \Lambda^s(\psi)$) is bounded.
\end{theorem}


\subsection{Proof of Theorem \ref{thm_criterion}} \label{subsec:SOE}

To set up for the proof, we begin by giving precise definitions of translation number with respect to $\varphi$, and the displacement of a curve by a Dehn twist. 

\begin{definition}[Translation number]
Let $\beta \in \pi_1(M)$, and consider its action on $\Lambda^s(\phi)$.  Fix $x \in \Lambda^s(\phi)$.  For each $q \in \bZ$ there exists a unique $p_q \in \bZ$ 
so that $\tau^{p_q}(x) \leq \beta^q(x) < \tau^{p_q+1}(x)$.  
We define 
\[ \trans_\phi(\beta) := \lim_{q \to \infty} \frac{p_q}{q}. \]
\end{definition}
Since $\tau$ and $\beta$ commute, it is a standard exercise to show that this limit exists and is in fact independent of the choice of $x$; indeed, $\trans_\phi(\beta)$
 is simply the classical {\em translation number} for the action of $\beta$ on $\Lambda_s$ with respect to any parametrization of $\Lambda_s$ where $\tau$ acts as translation by $1$. 
 
 \begin{rem} \label{rem_translation_number_and_lozenges}
  We will consider below only the case when $\beta$ is freely homotopic to a curve in a quasi-transverse torus $T$. Thus $\beta$ fixes a (unique) string of lozenges in the orbit space. In this case we therefore have, $\trans_{\phi}(\beta)= k \in \bZ$ where $k$ is such that if $x$ is any of the corner of the string of lozenges, then $\beta x = \tau^k(x)$. 
 \end{rem}

If a skew Anosov flow $\varphi$ on $M$ is transversally oriented, we saw (Lemma \ref{lem:klein}) that $M$ cannot admit an incompressible embedding of the Klein bottle, thus by  \cite[Th\'eor\`eme C]{Bar_MPOT} any incompressible embedded torus in $M$ can be put in quasi-transverse position with respect to the flow.
We further have, thanks to \cite[Th\'eor\`eme E]{Bar_MPOT}, that  any collection of pairwise disjoint, non-isotopic  incompressible embedded tori,  can be simultaneously isotoped to a collection of still disjoint (and obviously still non-isotopic) quasi-transverse tori. It is thus no loss of generality to adopt the following convention for transversely orientable flows.

\begin{convention*}
For the remainder of this section, we restrict our attention to transversely orientable flows, and we will always assume that the tori we consider are in quasi-transverse position.
\end{convention*}

Next we will define the ``displacement" of an orbit by a Dehn twist.  We first recall the definition of Dehn twists to emphasize that they come with a specification of a transverse orientation on the torus.

\begin{definition}
Let $T$ be an embedded torus, and $\beta \in \pi_1(T)$.  A {\em Dehn twist along $\beta$} is the mapping class of a map $D_\beta$ defined as follows:  Take a small product region $T \times [-1,1] \in M$, and fix a basis $\{ \alpha, \nu \}$ for $\pi_1(T)$ giving an identification of $T$ with $S^1 \times S^1 = \R/\bZ \times \R/\bZ$ where 
$\beta = \alpha^p \nu^q$ for some $p,q$.   For $(x,y, z) \in T \times [-1,1]$, we define $D_\beta(x,y,t) = (x + p h(t), y+ q h(t), t)$, where $h\colon [-1,1] \to [0,1]$ is a smooth bump function with $h(-1)=0$ and $h(1)=1$, and extend $D_\beta$ to be the identity elsewhere on $M$.
\end{definition} 

\begin{rem}
Note that reversing the transverse orientation of $T$ and applying the same construction results in a map isotopic to the {\em inverse} of that defined above.  Thus, the notation $D_\beta$, while standard, is somewhat misleading because the mapping class does not depend on $\beta$ alone.   There is no intrinsic way to
distinguish $D_\beta$ from $D_\beta^{-1}$.  Thus,  by convention, we say that a map $D_\beta$ {\em comes with} the data of a choice of transverse orientation.  When we speak of a Dehn twist supported on a torus neighborhood $T \times [-1,1]$, we always assume the orientation is as given by the interval $[-1,1]$.  
\end{rem} 

It will be convenient for us to choose homeomorphisms representing Dehn twists which are in a particularly nice form with respect to the flow, as follows.  

\begin{convention}[Good Dehn twist coordinates] \label{conv:coord}
 Given a quasi-transverse torus $T$, we choose the coordinates $(x,y, z) \in T \times [-1,1]$ in the following way: Let $\alpha_1, \dots, \alpha_{2n}$ be the periodic orbits of $\varphi$ on $T$. Then we assume that the local stable leaves of the orbits $\alpha_i$ in $T \times [-1,1]$ are given by the equation $x = i/2n$.
\end{convention}

With this convention, a Dehn twist $D_\alpha$ on $T$, where $\alpha \in \pi_1(T)$ represents any power of the periodic orbits, will preserve the local stable leaves of the periodic orbits $\alpha_i$.  More generally, given any Dehn twist $D_\beta$ on $T$, and any segment $c(t) := (x_0, y_0, t), t\in [-1,1]$ through $T \times [-1,1]$, the number of times its image $D_\beta(c(t))$ intersects the union of stable leaves of the $\alpha_i$ is the minimal intersection number of $\beta$ with $\alpha_i$.

\begin{definition}[Sign of an intersection] \label{def:sign} 
Given a Dehn twist $D_\beta$ supported on $T \times [-1, 1]$ as above, and an orbit $\phi_t$ intersecting $T$ transversely at $t=0$, we say this intersection is {\em positive} if 
$\varphi^t(z) \in T \times [0,1]$ for small positive $t$, and {\em negative} if $\varphi^t(z) \in T \times [-1,0]$ for small positive $t$.   
\end{definition}

Before giving the formal definition of the displacement $\disp_\varphi(c,f)$ of a closed orbit $c$ under a product of Dehn twists $f$, we motivate this with the following lemma, which describes how a Dehn twist on a torus $T$ affects a segment of a periodic orbit transverse to $T$, from the perspective of the leaf space of $\varphi$.  

For the statement, we fix a quasi-transvers torus $T$ in $M$, a point $z$ with orbit $\varphi^t(z)$ transverse to $T$, and a small product neighborhood $T \times [-1,1]$
so that the orbit $\varphi^t(z) \cap T \times [-1,1]$ is, locally near $t=0$, a segment $J$ between some point $z_- \in T \times \{-1\}$ and $z_+ \in T \times \{1\}$.  Let 
$D_\beta$ be a Dehn twist supported on $T \times [-1,1]$. 

\begin{lemma} \label{lem_action_of_one_dehn_lift}
Let $\wt T \times [-1,1]$ be a lift of $T \times [-1,1]$ to $\wt M$, let $\wt J$ be the lift of $J$ in $\wt T \times [-1,1]$ with endpoints $\tilde z_-, \tilde z_+$, and let $\wt D_\beta$ be the lift of $D_\beta$ fixing $\tilde z_-$.  Finally, let $\{L_i\}_{i\in \bZ}$ be the string of lozenges associated with $\wt T$.
\begin{enumerate}
\item 
If $\wt J$ projects to a point in the lozenge $L_i$, then $\wt D_\beta(\tilde z_+)$ projects into $L_k$, where $k = i + 2\trans_\varphi(\beta)$ if the intersection of $J$ and $T$ is positive, and $k = i-2\trans_\varphi(\beta)$ if the intersection is negative.
\item The stable saturation of $\wt D_\beta(\wt J)$ to the orbit space stays inside the stable saturation of the lozenges between $L_i$ and $L_k$.
\end{enumerate}
\end{lemma} 

\begin{proof}
Without loss of generality, we assume that $\wt J$ projects to a point in $L_0$, and that the sign of the intersection is positive (the negative case is analogous).  
Recall that the lozenges $L_i$ are the projections of strips $A_i$ inside $\wt T$, bounded by periodic orbits of $\wt \varphi$. 

By definition of $D_\beta$, the image $\wt D_\beta(\tilde z_+)$ projects in the orbit space to the lozenge $\beta(L_0)$. Now, if $x_i$ are the corners of the lozenge $L_i$ (enumerated so that $x_i < x_{i+1}$ in $\Lambda_s(\varphi)$), then $\eta^k(x_0) = x_{2k}$.
Thus, by definition of translation number (and Remark \ref{rem_translation_number_and_lozenges}), $\beta(L_0) = L_{2\trans_\varphi(\beta)}$, proving the first part of the lemma.

The second statement follows immediately from Convention \ref{conv:coord} and the structure of lozenges (Lemma \ref{lem:lozenges}).
\end{proof}

\begin{definition}[Displacement]
Let $c$ be a periodic orbit of the skew flow $\varphi$ and $D_\beta$ a Dehn twist on a quasi-transverse torus $T$.
Let $x_1, \dots, x_n$ be the intersection points of $c$ with $T$, and let $\epsilon_i = \{\pm 1\}$ be the sign of the intersection at $x_i$. 
The {\em displacement of $c$ by $D_\beta$ is defined by}
\[
 \disp_\varphi(c, D_\beta) = \sum_{i=1}^n 2\epsilon_i \trans_\varphi(\beta).
\]

If $f = D_{\beta_1} \circ \ldots \circ D_{\beta_k}$ is a composition of Dehn twists on \emph{pairwise disjoint} quasi-transverse tori, then we set
\[
\disp_\varphi(c, f) = \sum_{i=1}^k \disp_\varphi(c, D_{\beta_i}).
\]
\end{definition}

We will need the following observation for the proof of Theorem \ref{thm_criterion}.
\begin{observation}\label{obs:consec_lozenge}
Let $\wt T$ and $\wt T'$ be two \emph{disjoint} lifts of (the same or distinct) quasi-transverse tori. Let $\wt J$ be a segment of a periodic orbit that first crosses $\wt T$ and then crosses $\wt T'$.
Suppose $\wt T$ projects to a string of lozenges $\{L_i\}$, and $\wt T'$ to a string $\{L'_i\}$. Let $i,j$ be such that $\wt J$ is contained in $L_i \cap L'_j$.

Then $L'_j$ is contained in the saturation of $L_i$ by stable leaves; equivalently, $L_i$ is contained in the saturation of $L'_j$ by unstable leaves. 
\end{observation} 

\begin{proof} 

Since the lifts $\wt T$ and $\wt T'$ are disjoint, observe that either $L_i \subset \hfs(L'_{j})$ and $L'_{j}\subset \hfu(L_{i})$, or $L_i\subset \hfu(L'_{j})$ and $L'_{j}\subset \hfs(L_{i})$.  Both configurations are shown in Figure \ref{fig_consecutive_lozenges}.

\begin{figure}[h]
\begin{center}
\scalebox{0.85}{
\begin{pspicture}(-0.5,-0.5)(6,6)

\pspolygon[linestyle=none,fillstyle=solid,
fillcolor=lightgray!20!Lavender!10](1.5,3)(4,3)(4,2.5)(1.5,2.5)

\psline[linewidth=0.04cm,linestyle=dashed](2,0.5)(5,3.5)
\psline[linewidth=0.04cm,linestyle=dashed](0.5,2)(3.5,5)
\psline[linewidth=0.04cm,arrowsize=0.05cm 2.0,arrowlength=1.4,arrowinset=0.4]{->}(1,0)(5,0)
\psline[linewidth=0.04cm,arrowsize=0.05cm 2.0,arrowlength=1.4,arrowinset=0.4]{->}(0,1)(0,5)
\rput(5.2,-0.2){$\Lambda^u$}
\rput(-0.2,5.2){$\Lambda^s$}
 \psline[linewidth=0.04cm,linecolor=blue](0.5,2)(3.5,2)
 \psline[linewidth=0.04cm,linecolor=red](2,0.5)(2,3.5)
\psline[linewidth=0.04cm,linecolor=red](3.5,2)(3.5,5)
\psline[linewidth=0.04cm,linecolor=blue](2,3.5)(5,3.5)

\psline[linewidth=0.04cm,linecolor=blue,linestyle=dashed](1,2.5)(4,2.5)
\psline[linewidth=0.04cm,linecolor=blue,linestyle=dashed](1.5,3)(4.5,3)
\psline[linewidth=0.04cm,linecolor=red,linestyle=dashed](1.5,3)(1.5,0.5)
\psline[linewidth=0.04cm,linecolor=red,linestyle=dashed](4,2.5)(4,5)

\psdot[linewidth=0.04cm](2.75,2.75)
\uput{3pt}[-155](2.75,2.75){$\widetilde J$}

\psdot[linewidth=0.04cm](1.5,2.5)
\psdot[linewidth=0.04cm](2,2)
\psdot[linewidth=0.04cm](4,3)
\psdot[linewidth=0.04cm](3.5,3.5)
\end{pspicture}
}%
\scalebox{0.85}{\begin{pspicture}(-0.5,-0.5)(6,6)

\pspolygon[linestyle=none,fillstyle=solid,
fillcolor=lightgray!20!Lavender!10](3,1.5)(3,4)(2.5,4)(2.5,1.5)
\psline[linewidth=0.04cm,linestyle=dashed](2,0.5)(5,3.5)
\psline[linewidth=0.04cm,linestyle=dashed](0.5,2)(3.5,5)
\psline[linewidth=0.04cm,arrowsize=0.05cm 2.0,arrowlength=1.4,arrowinset=0.4]{->}(1,0)(5,0)
\psline[linewidth=0.04cm,arrowsize=0.05cm 2.0,arrowlength=1.4,arrowinset=0.4]{->}(0,1)(0,5)
\rput(5.2,-0.2){$\Lambda^u$}
\rput(-0.2,5.2){$\Lambda^s$}
 \psline[linewidth=0.04cm,linecolor=blue](0.5,2)(3.5,2)
 \psline[linewidth=0.04cm,linecolor=red](2,0.5)(2,3.5)
\psline[linewidth=0.04cm,linecolor=red](3.5,2)(3.5,5)
\psline[linewidth=0.04cm,linecolor=blue](2,3.5)(5,3.5)

\psdot[linewidth=0.04cm](2.75,2.75)
\uput{3pt}[90](2.75,2.75){$\widetilde J$}

\psline[linewidth=0.04cm,linecolor=blue,linestyle=dashed](0.5,1.5)(3,1.5)
\psline[linewidth=0.04cm,linecolor=blue,linestyle=dashed](2.5,4)(5,4)
\psline[linewidth=0.04cm,linecolor=red,linestyle=dashed](2.5,1)(2.5,4)
\psline[linewidth=0.04cm,linecolor=red,linestyle=dashed](3,1.5)(3,4.5)

\psdot[linewidth=0.04cm](2.5,1.5)
\psdot[linewidth=0.04cm](2,2)
\psdot[linewidth=0.04cm](3,4)
\psdot[linewidth=0.04cm](3.5,3.5)
\end{pspicture}
}%
\end{center}
 \caption{The configuration on the right is impossible, if the shaded lozenge is $L'_{j}$} \label{fig_consecutive_lozenges}
\end{figure}
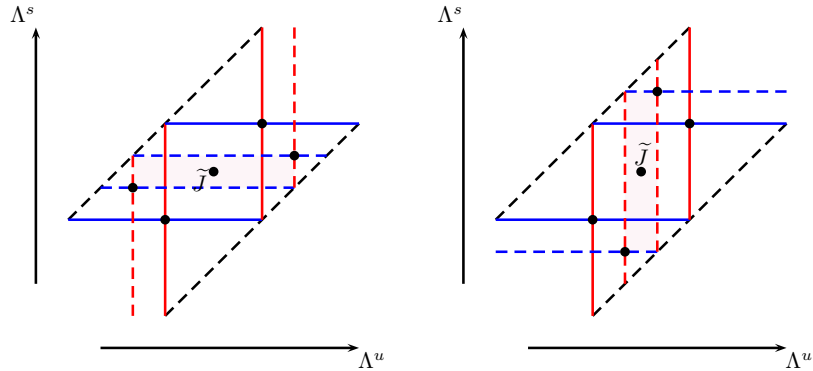

We will show that the second case cannot occur because orbits converge in the positive direction of the flow along $\wt J$.  Precisely, consider the strong unstable leaf through the starting point $\wt x$ of the orbit segment $\wt J$.  Denote this leaf by $\lambda^{uu}$, and
let $l_i \subset \lambda^{uu}$ denote the set of points whose forward orbit intersects the strip $A_i$ in $\wt T$ corresponding to $L_i$. Similarly, define $l'_j \subset \lambda^{uu}$ to be the set of points whose forward orbit intersects the strip $A'_j$ in $\wt T'$ corresponding to $L_j'$.
Hyperbolicity of the flow means that $l'_{j} \subset l_{i}$. Hence $L_j' \subset \hfs(L_{i})$, and equality happens if and only if $L_j'= L_i$ which in turns imply that the orbits bounding $A_i$ and $A_j'$ are the same, which contradicts the assumption that $\wt T$  and $\wt T'$ are quasi-transverse and disjoint. This proves the observation.
\end{proof}

We next reduce the proof of Theorem \ref{thm_criterion} to the following proposition.  

\begin{proposition}\label{prop_bounded}
Suppose $f = D_{\beta_1} \circ \ldots \circ D_{\beta_n}$ is a composition of Dehn twists on pairwise disjoint quasi-transverse tori $T_1, \ldots, T_n$, and let $\tilde{c} \subset \tilde{M}$ be a lift of a periodic orbit $c$ of $\varphi$. 
  Then $\disp_\phi(c,f) = 0$ if and only if for some lift (equivalently, for all lifts) $\tilde{f}$ of $f$ the image 
  $\tilde{f}(\tilde{c})$ has bounded projection to $\Lambda^s(\varphi)$.
\end{proposition} 

\begin{proof}[Proof of Theorem \ref{thm_criterion} given Proposition \ref{prop_bounded}]
Indeed, this is an easy consequence of Proposition \ref{prop_charac_periodic_orbits} and Theorem \ref{thm_main}.  
Theorem \ref{thm_main} states that $f$ is a self orbit equivalence if and only if $\mathcal{P}(\varphi) = f_\ast \mathcal{P}(\varphi) = \mathcal{P}(f \varphi f^{-1})$.
By Proposition \ref{prop_charac_periodic_orbits}, $\mathcal{P}(\varphi)$ can be characterized as the set of $\gamma \in \pi_1(M)$ such that every $\gamma$-invariant curve $\tilde{c} \subset \tilde{M}$ has bounded projection to $\Lambda^s(\varphi)$.   Thus, $f$ is isotopic to a self-orbit equivalence if and only if $f_*$ preserves this set. 
Proposition \ref{prop_bounded} characterizes this set in terms of the vanishing of $\disp_\phi(c,f)$.
\end{proof}

\begin{proof}[Proof of Proposition \ref{prop_bounded}]
Fix $f = D_{\beta_1} \circ \ldots \circ D_{\beta_k}$,
where $D_{\beta_i}$ is a Dehn twist on $T_i$, and let $\tilde{f}$ be a lift of $f$ to $\tilde{M}$ chosen so that $\tilde{f}$ fixes some point $\tilde{x}$ on $\tilde{c}$. Recall that we assumed, without loss of generality, that all the $T_i$ are quasi-transverse.

Let $\gamma \in \pi_1(M)$ represent a periodic orbit $c$ of $\varphi$.
As a first trivial case, suppose $c$ has no transverse intersections with any torus $T_i$, so it is either contained in a single torus or disjoint from all of them.  Let $\tilde{c}$
be a lift of $c$.  Note that $c$ is isotopic to a curve $c'$ disjoint from the support of $f$; lifting this isotopy means the lift $\tilde{c}$ is isotopic to a lift $\tilde{c'}$ such that, after applying some deck transformation $g$ of the cover, this is disjoint from the support of $\tilde{f}$.  Thus, $\tilde{f} (g \tilde{c})$ is uniformly bounded distance away from 
$\tilde{f}( g \tilde{c'}) = g\tilde{c'}$, and this is uniformly bounded distance away from $g \tilde{c}$, showing that $g \tilde{c}$ (and hence also $\tilde{c}$) has bounded projection.  

For the case where $c$ intersects some $T_i$ transversely, we may without loss of generality choose the product neighborhoods of the tori $T_i$ in the definition of $D_{\beta_i}$ small enough so that every time $c$ crosses a torus $T_i$, it enters one side $T_i \times \{\pm 1\}$ and leaves the other,  $T_i \times \{\mp 1\}$.  
Consider the positive-time ray $r = \left\{\varphi^t(\tilde{x}) , t\geq 0\right\} \subset \wt c$, where $\tilde{x}$ is as before a point fixed by $\tilde{f}$.  We will show that the projection of $\wt f(r)$ to $\Lambda^s(\varphi)$ is bounded. Reversing the argument (using unstable leaves instead of stable),
will show that the projection of $\wt f(\wt c)$ to $\Lambda^s(\varphi)$ is also bounded, so we do only the forward case.  

Between $\wt x$ and $\gamma \wt x$, the ray $r$ intersects a finite number of lifts of the tori $T_i$ on which the Dehn twists $D_{\beta_i}$ are supported. Let
$\cT_1, \cdots, \cT_n$ denote these lifts, indexed along the path of $r$ so that $r$ first intersects $\cT_1$ after $\wt x$.  Note that two distinct lifts $\cT_i$ and $\cT_j$ may project to the same torus  in $M$, this will happen whenever the orbit $c$ crosses the same torus twice. 
Since the $T_i$ are quasi-transverse tori, each $\cT_j$ projects to a string of lozenges $\cC_j$. Let $L_0^{(1)}$ be the lozenge in $\cC_1$ containing the projection of $\wt c$.
The main technical part of the proof is the following claim.

\begin{claim} \label{claim_zero_weighted_intersection}
Let $r_0$ denote the segment of $r$ between $\wt x$ and $\gamma \wt x$.  
\begin{enumerate}
\item \label{item_first_part_of_claim}  The stable leaf of $\wt f (\gamma\wt x)$ intersects the lozenge $\eta^{\disp_\varphi(c, f)}(L_0^{(1)})$.
\item \label{item_second_part_of_claim} There exists $N>0$, depending only on $f$, such that the stable leaf saturation of $\wt f( r_0)$ intersects the chain of lozenges $\cC_1$ only between $\eta^{-N}(L_0^{(1)})$ and $\eta^{N}(L_0^{(1)})$.
\end{enumerate} 
\end{claim}

Given this claim, the proof of Proposition \ref{prop_bounded} can be finished quickly, by considering the positive iterates of $r_0$ under $\gamma$.  
Let us assume the claim for the moment and use it to derive the conclusion of Proposition \ref{prop_bounded}.  

We use the fact that $r = \bigcup_{i=1}^\infty \gamma^i(r_0)$.   For the first direction, suppose that $\disp_\varphi(c, f)=0$.   Fix a segment $r_i = \gamma^i(r_0)$.  Then $\gamma^{-i}(r_i) = r_0$ so by Claim \ref{claim_zero_weighted_intersection}, the stable leaf of $\wt f (r_i) = f_*(\gamma^i) \wt{f}(r_0)$ intersects the lozenge $f_*(\gamma^i) (L_0^{(1)})$, and $\wt f( r_i)$ intersects the chain of lozenges $\cC_1$ only between $f_*(\gamma^i) \eta^{-N}(L_0^{(1)})$ and $f_*(\gamma^i)  \eta^{N}(L_0^{(1)})$.  But Observation \ref{obs:consec_lozenge} implies that $f_*(\gamma^i)(L_0^{(1)})$ is contained in the stable saturation of $L_0^{(1)}$, thus, $\wt f (r_i)$ is contained in the union of the stable saturation of leaves between $\eta^{-N}(L_0^{(1)})$ and $\eta^{N}(L_0^{(1)})$, for all $i$, hence $r$ has bounded projection to $\Lambda^s(\varphi)$.   As remarked above, applying the same argument to unstable leaves and using the negative time ray shows that $c$ has bounded projection.  

Conversely, if $\iota=i_\varphi(c, f)\neq 0$, then the argument above shows that $\hfs\left(\wt f (\gamma^n\wt x)\right)$ intersects $\eta^{n\iota}(L_1^{(1)})$, thus the projection of $\wt f(r)$ to $\Lambda^s(\varphi)$ is unbounded.
\end{proof} 
So to finish the proof of the Proposition, we only need to prove Claim \ref{claim_zero_weighted_intersection}.  

  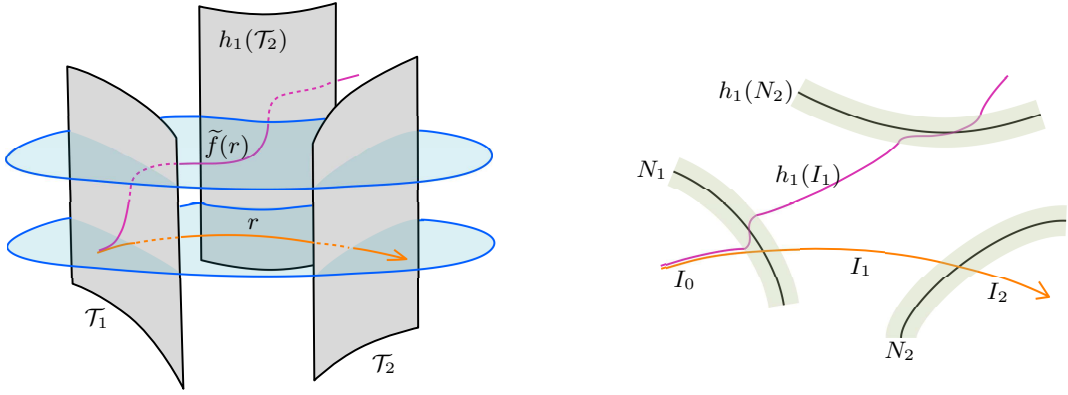
\begin{figure}
   \labellist 
  \small\hair 2pt
     \pinlabel $\cT_1$ at 80 65 
    \pinlabel $\cT_2$ at 300 30 
    \pinlabel $h_1(\cT_2)$ at 200 280 
    \pinlabel $r$ at 200 138 
   \pinlabel $\wt{f}(r)$ at 180 200 
   \pinlabel $I_0$ at 530 95 
   \pinlabel $N_1$ at 505 180
   \pinlabel $N_2$ at 695 40
   \pinlabel $h_1(N_2)$ at 585 240
   \pinlabel $I_1$ at 665 105
   \pinlabel $h_1(I_1)$ at 625 175
   \pinlabel $I_2$ at 770 85
   \endlabellist
     \centerline{ \mbox{
 \includegraphics[width = 5.5in]{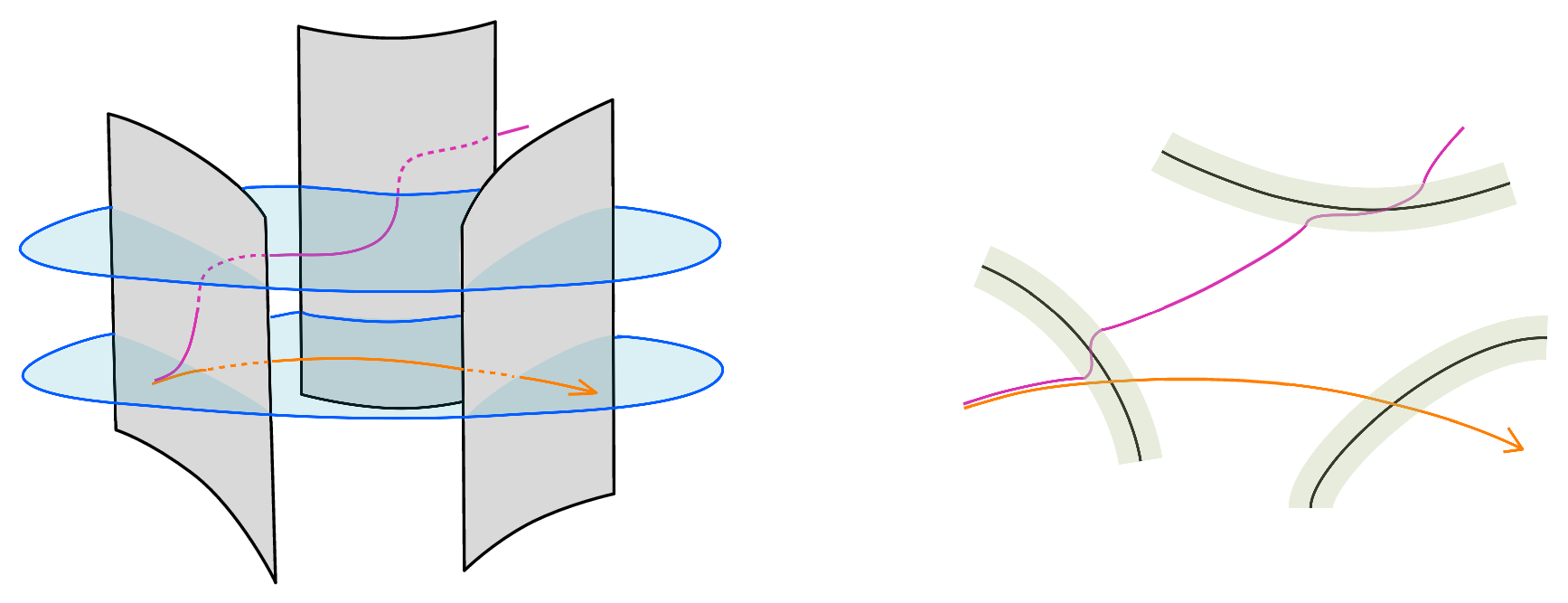}}}
 \caption{The action of the lift of a Dehn twist on a ray, shown in $\wt M$ on the left and schematically indicating intersections with lifted neighborhoods of tori on the right.}
  \label{fig:self-sim}
  \end{figure}

\begin{proof}[Proof of Claim \ref{claim_zero_weighted_intersection}]
Recall that $f$ has support in disjoint small neighborhoods of the quasi-transverse tori in $M$ which lift to disjoint neighborhoods of $\cT_j$ in $\wt M$, and we use $L^{(j)}_k$ to denote the string of lozenges associated to $\cT_j$.   Let $\wt N_j$ denote the lifted neighborhood containing $\cT_j$. We use these neighborhoods to split the ray $r_0$ into a union of intervals $I_j$ and $J_j$ where $J_j := \wt N_j \cap r_0$, and $I_j$ denotes the connected component of $r_0$ between $ \wt N_{j-1}$ and  $\wt N_j$.  The interval $I_0$ is the segment of $r_0$ from $\wt x$ to $\wt N_1$.

For each  $\cT_i$, let $t_i$ denote the translation number of the Dehn twist on the corresponding torus.  
Notice that, since $f$ is supported on the union of the projection of the neighborhoods $\wt N_i$ to $M$, it follows that for any $j$, there exists $h_j\in \pi_1(M)$ such that $\wt f(I_j) = h_j (I_j) \subset h_j \wt c$. Our choice of lift implies that $h_0$ is the identity.    

To demonstrate the first statement of the claim, we will first establish the fact that, if $\wt f(I_{i-1}) = h_{i-1} (I_{i-1})$ intersects a lozenge $h_{i-1} L^{(i)}_{k}$, then 
the stable leaf of $\wt f(I_{i})$ intersects the translate of $h_{i-1} L^{(i)}_{k}$ by $\eta^{\epsilon_{i} 2 t_{i}}$.  Here is the proof of the fact:  Fix some $i$.  The segment $\wt f(J_i)$ is obtained from $h_{i-1}(J_i)$ by applying the (unique) lift of the Dehn twist supported on the projection of $\cT_i$ that preserves $h_{i-1} \cT_i$ and fixes $h_{i-1}(I_{i-1})$.    
The (projection to the orbit space of the) endpoint shared by $\wt f(I_{i-1}) = h_{i-1}(I_{i-1})$ and $\wt f (J_i)$ lies in some lozenge $L^{(i-1)}_k$ associated to $h_{i-2}( \cT_{i-1})$.  By Observation \ref{obs:consec_lozenge}, the lozenge $h_{i-1}( L^{(i)}_{k'})$ for $h_{i-1}( \cT_i)$ that also contains this point is such that $h_{i-1} (L^{(i)}_{k'})$ is contained in the saturation by stable leaves of $h_{i-2} (L^{(i-1)}_{k})$.
Applying Lemma \ref{lem_action_of_one_dehn_lift}, we conclude that after applying the lifted Dehn twist to $h_{i-1}(J_i)$ (see Figure \ref{fig:self-sim}), the image of its other endpoint, which is also an endpoint of $\wt f(I_i)$, lies in $\eta^{\epsilon_i 2 t_i} h_{i-1}( L^{(i))}_{k'}$.  This is contained in the saturation by stable leaves of $\eta^{\epsilon_i 2 t_i} h_{i-2}( L^{(i-1)}_{k})$, which proves the fact.

To deduce (\ref{item_first_part_of_claim}) using the fact, we apply it iteratively, observing first that the stable leaf of $\wt f (I_1)$ intersects $\eta^{\epsilon_1 2 t_1} (L^{(1)}_0)$ and some lozenge $h_1 (L^{(2)}_k)$, and the stable saturation of $L^{(2)}_k$ is contained in $\eta^{\epsilon_1 2 t_1} L^{(1)}_0$.  Thus, $\wt f (I_2)$ has stable leaf intersecting $\eta^{\epsilon_2 2 t_2} h_1 (L^{(2)}_k)$, and thus $\eta^{\epsilon_1 2 t_1} \eta^{\epsilon_2 2 t_2} (L^{(1)}_0)$.  Continuing iteratively, we conclude that $\wt f(\gamma \wt x)$ has stable leaf intersecting the translate of $L^{(1)}_0$ by $\eta^{d_\phi(c, f)}$.

Now, the second part of the claim follows directly from the application of the second part of Lemma \ref{lem_action_of_one_dehn_lift} in the proof above by choosing 
\[
 N = \sum_{i=1}^k 2|t_i|. \qedhere
\] 
\end{proof}

\subsection{Application: classification of self-orbit equivalences in special cases} \label{sec:special_cases} 

We will now use the criterion given by Theorem \ref{thm_criterion} to describe the isotopy class of self orbit equivalences in some special cases, starting with the proof of Corollaries \ref{cor_horizontal_twists} and \ref{cor_non_horizontal_twists}.

\begin{proof}[Proof of Corollary \ref{cor_horizontal_twists}]
 Suppose that $D_\alpha$ is a Dehn twist on a torus $T$ in a direction $\alpha$ that is represented by a periodic orbit of $\varphi$. Then the translation number of $\alpha$ is zero, and thus, for any periodic orbit $c$ of $\varphi$, the displacement of $c$ by $D_\alpha$ is zero. Therefore, the criterion of Theorem \ref{thm_criterion} implies that any such Dehn twist is isotopic to a self orbit equivalence of $\varphi$, which proves the corollary.
\end{proof}

\begin{proof}[Proof of Corollary \ref{cor_non_horizontal_twists}]
 Let $T$ be a quasi-transverse torus in $M$ and $D_\beta$ a Dehn twist on $T$ with nonzero translation number. 
 Suppose first that $T$ is separating. Consider a periodic orbit $c$ of $\varphi$. Then either $c$ does not intersect $T$ or it intersects $c$ an even number of times with alternating signs.  In either case, $\disp_\varphi(c,D_\beta)= 0$. Thus $D_\beta$ is isotopic to a self orbit equivalence by Theorem \ref{thm_criterion}.

 Now, if we assume instead that $T$ is non-separating, we claim that we can find a periodic orbit $c$ of $\varphi$ such that its intersections with $T$ are always in the same transverse direction, and therefore are assigned the same sign. To do this, take a closed oriented loop based in $T$ in $M$ intersecting $T$ exactly once, transversely, at the basepoint (and therefore inducing a transverse orientation of $T$).   Lift this loop to a path in $\wt M$ with endpoints in lifts $\wt T_0$ and $\wt T_1$ of $T$, and iteratively choose successive lifts $\wt T_i$.   Then for each $i$, $\wt T_i$ separates $\wt T_{i-1}$ from $\wt T_{i+1}$, and no lifts of $T$ separates $\wt T_{i}$ from $\wt T_{i+1}$.   
 
Fix $n$ large, and let $\wt d$ be a segment of an orbit in $\wt M$ with endpoints in $\wt T_0$ and $\wt T_n$.  Such an orbit always exists because the flow is $\bR$-covered and skew. 
If $n$ is chosen large enough, $\wt d$ projects down to an orbit segment in $M$ that will contain points close enough to satisfy the conditions of the Anosov closing lemma. Hence we obtain a periodic orbit $c$ such that all its intersections with $T$ have the same transverse orientation.  We conclude that $\disp_\varphi(c,D_\beta)$ is non-zero, which, by the criterion implies that $D_\beta$ is not isotopic to a self orbit equivalence.
\end{proof}

The proof of this corollary gives an example of a general strategy to identify self orbit equivalences by understanding the intersection of periodic orbits with quasi-transverse tori.   One could broaden this to give a characterization of self orbit equivalences of a given flow in terms of the configuration of tori in $M$ (the JSJ graph and geometry of the pieces) and translation numbers of periodic orbits in tori.   Working through the details of this is beyond the scope of this article; instead we illustrate our criterion with some examples where we can virtually describe the group of self orbit equivalences in the mapping class group.

For both of the following statements, we assume $\varphi$ is a transversally orientable skew Anosov flow on $M$.

\begin{theorem} \label{thm_oneJSJpiece}
Suppose that the JSJ decomposition of $M$ has a single piece which is atoroidal, glued to itself along $n$ boundary tori $T_1, \dots, T_n$. 
 Let $D_{\mathrm{per}}$ be the group generated by Dehn twists on the $T_i$ in the directions of the periodic orbits of $\varphi$.
 
Let $D$ be the finite index subgroup of the mapping class group generated by Dehn twists on tori.  Then $D_{\mathrm{per}}$ is the set of self orbit equivalences in $D$.
\end{theorem}

\begin{theorem} \label{thm_linear_graph}
 Suppose that each torus of the JSJ decomposition of $M$ is separating.  Let $D$ be the finite index subgroup of the mapping class generated by Dehn twists on tori.  Then the set of self orbit equivalences in $D$ is equal to the group generated by Dehn twists in the directions of periodic orbits inside Seifert pieces, together with any Dehn twists on the JSJ tori.
\end{theorem}

\begin{rem}
Using Handel-Thurston surgery and Foulon--Hasselblatt Dehn surgeries on periodic orbits of geodesic flows, one can easily construct many examples of skew Anosov flows on manifolds as in the statements of Theorem \ref{thm_oneJSJpiece} and   \ref{thm_linear_graph}.
\end{rem}

The proofs of both statements use the fact that the group generated by Dehn twists on tori has finite index in the mapping class group of any orientable $3$-manifold. (See \cite{JoBook} for explanation and history of the proof).  

\begin{proof}[Proof of Theorem \ref{thm_oneJSJpiece}]
 By Corollary \ref{cor_horizontal_twists}, any isotopy class $[f]\in D_{\mathrm{per}}$ is represented by a self orbit equivalence.
 For the other containment, suppose that $f$ is a self orbit equivalence contained in $D$.  
We assume for a contradiction that $f$ cannot be written as a product of elements in $D_{\mathrm{per}}$. Since elements of $D_{\mathrm{per}}$ are self orbit equivalences, up to composition with such a map we can assume that $f = D_{\beta_1} \circ \dots \circ D_{\beta_k}$, where $\trans_{\varphi}(\beta_i) \neq 0$ for all $i$.  In fact, we will only use that one of these has nonzero translation number.  

We now produce a periodic $c$ such that $\disp(c,f) \neq 0$, leading to a contradiction by Theorem \ref{thm_criterion}.
 Reindexing if needed, suppose $T_1$ is a torus in the support of $f$, with associated Dehn twist $D_{\beta_1}$. By hypothesis, $T_1$ is non separating in $M \smallsetminus \left(T_2 \cup \dots \cup T_n\right)$. By considering lifts of curves as in the proof of Corollary \ref{cor_non_horizontal_twists}, one can find a periodic orbit $c$ of $\varphi$ so that intersects only $T_1$, and each intersection point has the same orientation.
 Thus, $\disp(c,f) = \disp(c, D_{\beta_1}) = \pm \trans_{\varphi}(\beta_1) \neq 0$, contradicting Theorem \ref{thm_criterion}.
\end{proof}

\begin{proof}[Proof of Theorem \ref{thm_linear_graph}]
 Given the assumption, all JSJ tori are separating, so Corollary \ref{cor_non_horizontal_twists} show that any product of Dehn twists along them is isotopic to a self orbit equivalence. This, together with Corollary \ref{cor_horizontal_twists}, shows that any element of the group generated by Dehn twists in the direction of periodic orbits together with Dehn twists on the JSJ tori is isotopic to a self orbit equivalence. 
 
We are left to show that any element of the finite index subgroup generated by Dehn twists that is a self orbit equivalence is of this form.  
Let $f$ be a self orbit equivalence generated by Dehn twists, thus $f$ preserves each Seifert piece.  Fix a seifert piece $S$ and 
let $\Sigma$ denote the base orbifold of $S$, and $g$ the restriction of $f$ to $\Sigma$.  

Up to composing $f$ with some Dehn twists in the direction of periodic orbits on $S$ and Dehn twists on the boundary of $S$, we can assume that $g$ projects to the identity in the mapping class group of the base orbifold $\Sigma$ (see \cite{JoBook}, or \cite[section 4.7]{BGHP}).  
Thus, for any $\gamma \in \pi_1(S)$ that is obtained as a lift of an element of $\pi_1(\Sigma)$ that is not homotopic to a boundary component of $\Sigma$, there exists $k \in \mathbb{Z}$ such that $g_\ast (\gamma) = \gamma s^k$, where $s$ is the element of $\pi_1(S)$ representing a regular fiber of the Seifert fibration of $S$.    

For any such $\gamma$, the group $H=\langle \gamma, s\rangle$ is a subgroup of $\pi_1(M)$ homeomorphic to $\bZ^2$, and hence there exists a (unique) string of lozenges $\cC$ in $\orb$ that is preserved by $H$.  The permutation action of $H$ on this string gives a homomorphism $H \to \bZ$, thus there exists some nontrivial $\gamma' \in H$ whose action on $\orb$ fixes each corner of $\cC$,
equivalently, the conjugacy class of $\gamma'$ is represented by the periodic orbits of $\varphi$ that project to the corners of $\cC$.   

Since the weak foliations are transverse to the Seifert fibration, $s$ acts as a translation on the corners of $\cC$ (see \cite{BarbotFenley_pA_toroidal}).
In particular, if $k \neq 0$, then $g_\ast(\gamma') = \gamma' s^k$ does not fix a point of $\orb$, hence is not represented by a periodic orbit of $\varphi$.  This contradicts the fact that $g$ was the restriction of a self orbit equivalence.  Thus, $k=0$, and since the choice of $\gamma$ was arbitrary, we conclude that $g$ is isotopic to the identity.

Applying this argument to the restriction of $f$ to each Seifert piece, we conclude that $f$ is a product of Dehn twists in the direction of periodic orbits and Dehn twists on the JSJ tori, as claimed.
\end{proof}

\begin{rem}
The proof of Theorem \ref{thm_linear_graph} gives a virtual characterization of the self orbit equivalences of a given flow that fix a Seifert piece and are isotopic to identity on its boundary: they are exactly the isotopy classes of the group generated by Dehn twists in the direction of periodic orbits in the Seifert piece.
\end{rem}

\section{Application to contact flows} \label{sec:contact}
We will now apply some results of contact geometry to obtain the reverse direction of Theorem \ref{thm_isotopy_contact}, as well as its corollaries.
Recall that a (co-orientable) contact structure $\xi$ is \emph{Anosov} if $\xi$ admits an Anosov Reeb flow, i.e., there exists a $1$-form $\alpha$ such that $\xi = \ker \alpha$ and the Reeb flow of $\alpha$ is Anosov.

We will use as a black box the theory of cylindrical contact homology. Introduced in a more general context by Eliashberg--Givental--Hofer in \cite{EGH}, it has been proven to be well-defined for dynamically convex contact structures by Hutchings and Nelson \cite{HN}. This context includes the case of Anosov contact forms: If the Reeb flow of a contact $1$-form $\alpha$ is Anosov, then it is automatically non-degenerate and, since Anosov flows have no contractible orbits, $\alpha$ is dynamically convex (see also, e.g.~\cite{MP}).
 Some of the fundamental results in this theory can be summarized as follows:
 \begin{theorem} 
 If $\alpha_1$ and $\alpha_2$ are nondegenerate dynamically convex contact forms on a closed, hypertight, contact 3-manifold $(M ,\xi)$, then $C\bH^{\Lambda}_{\text{cyl}}(\alpha_1) \cong C\bH^{\Lambda}_{\text{cyl}}(\alpha_2)$ for any set $\Lambda$ of free homotopy classes in $M$.  The space $C\bH^{\Lambda}_{\text{cyl}}(\alpha_i)$ is a $\mathbb{Q}$-vector space, it is the homology of a complex generated by the periodic orbits of the Reeb flow of $\alpha_i$ in the free homotopy classes belonging to $\Lambda$. 
 \end{theorem}
 
 Having the cylindrical contact homology groups associated to $\alpha_1$ and $\alpha_2$ being isomorphic is not quite enough for our purpose: what we will need in order to apply Theorem \ref{thm_main} is to obtain that the chain complexes themselves are equal, when $\alpha_1$ and $\alpha_2$ are Anosov contact forms. This follows from the fact that the differential in the chain complex is trivial. That result was proven by Macarini and Paternain \cite[Theorem 2.1]{MP} for any Anosov contact structure (and independantly by Vaugon, see \cite{FHV}, with the additional assumption that the foliations are orientable).

\begin{proof}[Proof of Theorem \ref{thm_isotopy_contact}, reverse implication]
 Let $\varphi_1$, $\varphi_2$ be two contact Anosov flows with respective contact forms $\alpha_1$ and $\alpha_2$ and contact structures $\xi_1= \ker \alpha_1$ and $\xi_2 = \ker \alpha_2$. Suppose that $\xi_1$ and $\xi_2$ are contactomorphic. Then there exists a diffeomorphism $g\colon M \to M$ such that $g_\ast \xi_1 = \xi_2$.

In our situation, the $1$-forms $g^\ast \alpha_1$ and $\alpha_2$ are two contact forms of $(M,\xi_2)$ with Anosov Reeb flows. Applying the theorem on cylindrical contact homology above, where $\Lambda$ runs over all possible free homotopy classes in $\pi_1(M)$, we deduce that each homotopy class represented by a periodic orbit of the Reeb flow of $g^\ast \alpha_1$, is also represented by a periodic orbit of the Reeb flow, $\varphi_2$.

The Reeb flow of $g^\ast \alpha_1$ is $g^{-1}\circ \varphi_1^t \circ g$.  
 Thus, by Theorem \ref{thm_main}, $g^{-1}\circ \varphi_1 \circ g$ and $\varphi_2$ are isotopically equivalent. Equivalently, $\varphi_1$ and $\varphi_2$ are orbit equivalent.
 If additionally $g$ was isotopic to the identity, we conclude that $\varphi_1$ and $\varphi_2$ are isotopically equivalent.  This proves the theorem. 
\end{proof}

Using this direction of Theorem \ref{thm_isotopy_contact}, together with the coarse classification of contact structures in dimension 3 of Colin--Giroux--Honda, we can quickly deduce Theorem \ref{thm:finite_contact}.
\begin{proof}[Proof of Theorem \ref{thm:finite_contact}]
 By \cite[Th\'eor\`eme 6]{CGH}, on any irreducible $3$-manifold, there exist at most finitely many non-contactomorphic contact structures with bounded Giroux torsion.
 Now Proposition \ref{prop_zero_torsion} shows that the contact structure of any contact Anosov flow has zero Giroux torsion. Thus Theorem \ref{thm_isotopy_contact} implies the result\footnote{When $M$ is hyperbolic, one does not need Bowden's result, thanks to \cite[Th\'eor\`eme 2]{CGH}}.
\end{proof}

 
\appendix
\section{Further applications to contact topology\\ With \textsc{Jonathan Bowden}}
In this appendix, we show that any Anosov contact structure has zero Giroux torsion, prove the converse to Theorem \ref{thm_isotopy_contact}, and then use this to obtain new results about contact structures.  

\begin{proposition}\label{prop_zero_torsion}
 Let $\xi$ be an Anosov contact structure on a $3$-manifold $M$. Then $\xi$ has zero Giroux torsion. In fact, a double cover of $M$ will be strongly symplectically fillable.
\end{proposition}

\begin{proof} 
Let $\alpha$ be a contact $1$-form such that $\ker(\alpha) = \xi$ and such that the Reeb flow $\varphi$ of $\alpha$ is Anosov.  Let $X$ be the Reeb vector field.
Consider the $4$-manifold $M \times [-1, 1]$ with symplectic form $\omega = d(e^t\alpha)$ where $t$ is the coordinate on $[-1,1]$.
Then $\xi = \ker (\alpha)$ is a contact structure that we can assume without loss of generality to be positive and, since $\omega|_{M\times \{1\}} = e^1 d\alpha$, the form $\omega$ is non zero on $\xi$.

Up to taking a double cover, we can assume that $\varphi$ is transversally orientable, i.e.~its Anosov splitting is orientable.

Let $X^{ss}$ and $X^{uu}$ be vector fields in, respectively, the stable and unstable direction of the Anosov splitting of $X$, with orientations chosen so that the plane spanned by  $X^{ss}+ X^{uu}$ and $X$ defines a (co-orientable) contact structure $\xi_-$ with {\em negative} orientation (see \cite{Mitsumatsu}).  Put this contact structure on $M\times \{-1\}$.  

Since the defining property of a contact form is open in the $C^1$ topology, 
we can take a $C^1$-small approximation $\tilde\xi_-$ of $\xi_-$ such that $X$ is transverse to $\tilde\xi_-$.
Then, $\omega$ is non zero on $\tilde\xi_-$ (since $\omega|_{M\times \{-1\}} = e^{-1} d\alpha$, so its kernel is spanned by $X$).

This gives us a weak semi-filling 
of $(M\times \{1\}, \xi)$. By \cite[Corollary 1.4]{Eliashberg}, a weak semi-filling can be capped off to give a weak filling.  Since $\omega$ is exact, by \cite[Proposition 4.1]{Eliashberg}, this weak filling can be modified to give a strong filling (that result was independantly obtained by Etnyre in \cite{Etnyre}).
Now, Gay \cite[Corollary 3]{Gay} proved that any contact plane field that is strongly fillable has zero Giroux torsion. This proves the proposition for possibly a double cover of $M$.

Now the Giroux torsion of a contact structure cannot decrease under finite covers, since any component of a finite cover of a Giroux torsion domain is again a Giroux torsion domain. So, in any case, a contact structure $\xi$ has zero Giroux torsion if and only if any finite lifts of it has zero Giroux torsion. This proves the proposition for the original $\xi$.
\end{proof} 

Thanks to Giroux's correspondence between open books and contact structures \cite{Giroux}, we can prove the following, giving the second implication needed for the statement of Theorem \ref{thm_isotopy_contact}.  

\begin{theorem}\label{thm_converse_contact}
Let $\varphi_1$ and $\varphi_2$ be two contact Anosov flows with respective contact structures $\xi_1$ and $\xi_2$.
If $\varphi_1$ and $\varphi_2$ are orbit equivalent (resp.~isotopically equivalent) then $\xi_1$ and $\xi_2$ are contactomorphic (resp.~isotopic).
\end{theorem}

To apply Giroux's result in our proof, we first need to recall a result about Birkhoff sections, starting with their definition. 
\begin{definition}
Let $\varphi$ be a flow on a $3$-manifold $M$. A surface $S$ is called a {\em topological Birkhoff section} of $\varphi$ if:
\begin{enumerate}[label=(\roman*)]
 \item $S$ is topologically immersed in $M$, and its interior is topologically embedded.
 \item the flow $\varphi$ is topologically transverse to the interior of $S$,
 \item each connected component of the boundary of $S$ consists of a periodic orbit of $\varphi$,
 \item every orbit of $\varphi$ intersects $S$, and
 \item the return time of $\varphi$ to the interior of $S$ is uniformly bounded above and away from zero.
\end{enumerate}

A topological Birkhoff section that is smoothly immersed is called a (smooth) Birkhoff section.
 
A Birkhoff section $S$ in an orientable manifold is called \emph{positive} if the orientations of all the  boundary orbits correspond to the orientation induced by the flow on the interior of $S$.
\end{definition}
Fried showed that any transitive Anosov flow has a Birkhoff section, which can in fact be taken to be embedded on the boundary as well. Moreover, \cite{BonattiGuelman} show that this section can be assumed to be {\em tame}, meaning that after an isotopy along flow lines one can assume that the restriction of the Birkhoff section to a small tubular neighborhood of any component of its boundary is a smooth helicoid.
Finally, Marty \cite{Marty} showed that any $\bR$-covered Anosov flow admits a \emph{positive} tame Birkhoff section. Moreover, an adaptation of Marty's proof can be seen to yield a section that is also embedded on the boundary \cite{Marty_private}\footnote{One could also run the proof below using \emph{rational} open books, which corresponds to positive Birkhoff sections with immersed boundaries, see \cite{BEV}.}. 
 Thus one obtains an open book supporting the contact structure. We recall:
\begin{definition}[Open Book]
Let  $B \subset M$ be an oriented link in a connected oriented manifold. Then an open book $(B,\theta)$ with binding $B$ is a fibration with connected oriented (non-compact) fibers $\theta\colon M \setminus B \to S^1$ such that in a neighborhood of each binding component the map is equivalent to the map given by projecting to the angular polar coordinate on $(D^2 \setminus \{0\}) \times S^1$, and the boundary of any fiber agrees with $B$ as an oriented link.
\end{definition}
\noindent The fibers of an open book are called {\em pages}.  Note that the tameness condition in \cite{BonattiGuelman} corresponds precisely to a Birkhoff section $S$ inducing an open book with binding the (oriented) periodic orbits $\partial S$.

\begin{definition}[Supporting Open Book]
An open book $(B,\theta)$ supports a (co-oriented) contact structure $(M,\xi)$ if there is a contact form $\alpha$ for $\xi$ such that 
\begin{itemize} 
\item The form $d\alpha$ is positive on the pages of $\theta$, which are oriented to be compatible with the binding;
\item The form $\alpha$ is positive on (each component of) $B$.
\end{itemize}
\end{definition}
\noindent The fundamental fact due to Giroux \cite{Giroux} is that any two contact structures supported by a {\em fixed} open book are isotopic through contact structures supported by the open book. This is essentially due to the fact that the above condition is convex, although some care is needed near the binding.  

In order to apply the above, we will use the following result of \cite{BonattiGuelman} 
\begin{lemma}[\cite{BonattiGuelman}, Lemma 4.16]\label{lem_smooth_Birkhoff}
 Let $S$ be a topological Birkhoff section of a flow $\varphi$, then $S$ can be isotoped along the orbits of $\varphi$ to a smooth tame Birkhoff section $S'$.
\end{lemma}

\begin{proof}[Proof of Theorem \ref{thm_converse_contact}]
 Let $h$ be an orbit equivalence between $\varphi_1$ and $\varphi_2$. 
 Up to conjugating $\varphi_1$ by a diffeomorphism in the isotopy class of $h$, we can assume that $\varphi_1$ and $\varphi_2$ are isotopically equivalent and $h$ is isotopic to the identity. Showing that $\xi_1$ and $\xi_2$ are isotopic for this new flow will imply that $\xi_1$ and $\xi_2$ are contactomorphic for the original one.

 Let $S$ be a smooth Birkhoff section for $\varphi_1$. Then $h(S)$ is a topological Birkhoff section for $\varphi_2$. By the lemma above, we can isotope $h(S)$ to a smooth tame Birkhoff section $S_2$ of $\varphi_2$. Thus the open book $(B_2,\theta_2)$ induced by $S_2$ supports $\xi_2$. Now we isotope $h$ to a smooth map $g$ relative to the boundary of the original Birkhoff section $S$. Then, the open book $(B_2,\theta_2)$ pulls back to an open book $(B',\theta')$ with page $S' = g^{-1}S_2$ that supports $\xi_1' = g^*\xi_2$. This smoothing can be arranged so that pre-images of pages agree with those of an open book $(B,\theta)$ coming from the original Birkhoff section $S$ near the binding. Then one notes that  the open book $(B',\theta')$ is isotopic to the original one and we deduce that the corresponding contact structures are contactomorphic. Since $h$ (and hence $g$) is isotopic to the identity, they are in fact isotopic.
 \end{proof}

\subsection{Applications to contact topology}

Now we can use Theorem \ref{thm_isotopy_contact} to translate in the language of contact topology some known results about Anosov flows.  As a first example, Theorem \ref{thm_isotopy_contact} implies that the examples of skew-Anosov flows on hyperbolic 3-manifolds built in \cite{BowdenMann} (which are all contact flows when done using Foulon--Hasseblatt contact surgery) have non-contactomorphic contact structures. Thus, we immediately obtain

\begin{theorem}\label{thm_N_contact_structures}
 For any $N\in \bN$ there exists an hyperbolic $3$-manifold with at least $N$ non-contactomorphic Anosov contact structures.
\end{theorem}

This result answers affirmatively a question raised in \cite{FHV}, see also \cite[Question 7.4]{BowdenMann}.

We can also translate Theorems \ref{thm_oneJSJpiece} and \ref{thm_linear_graph} to a description of contact transformation groups, as follows.    
For a $3$-manifold $M$ with Anosov contact structure $\xi$, we follow \cite{GirouxMassot} and denote by $\cD(M,\xi)$ the group of contact transformations of $(M,\xi)$. 
Thus, there is a natural inclusion of $\pi_0\cD(M,\xi)$ in $\MCG(M)$.

\begin{theorem}\label{thm_ContactMCG}
Let $\xi$ be a Anosov contact structure on $M$.
Suppose that either
\begin{enumerate}
 \item $M$ has a unique JSJ piece which is atoroidal, or 
  \item each torus of the JSJ decomposition of the manifold $M$ is separating. 
\end{enumerate}
In the first case, let $D_\xi$ denote the subgroup of $\MCG(M)$ generated by Dehn twists on the JSJ tori.  In the second case, let $D_\xi$ be the subgroup of $\MCG(M)$ generated by all Dehn twists in the directions of periodic orbits together with any Dehn twists on the JSJ tori. 

Then, any class $[f] \in D_\xi$ admits a representative $f\in \cD(M,\xi)$, and conversely, there exists $n\in \bN$ such that, for any class $[f] \in \pi_0\cD(M,\xi)$, we have $[f]^n \in D_\xi$.
\end{theorem}

\begin{rem}
 This result partially extends a theorem of Giroux and Massot \cite{GirouxMassot}, who obtained this for the case of Seifert fibered manifolds. Note that the result of Giroux and Massot is more precise, as ours only give a description of $\pi_0\cD(M,\xi)$ up to finite powers.
\end{rem}

\begin{rem}
 Notice that it is not necessary to know the Anosov Reeb flow in order to detect which Dehn twist is in a direction of a periodic orbit, by the following observation:
 Let $T$ be an embedded torus that is quasi-transverse to the Anosov flow $\varphi$. Up to an arbitrarily small perturbation, one can put $T$ in a convex position with respect to the contact structure $\xi$. Then an element $\alpha\in \pi_1(T)$ corresponding to a periodic orbit of the flow $\varphi$ also corresponds to the free homotopy class of a connected component of the dividing set of the characteristic foliation of $\xi$ on $T$. Therefore, one can use Theorem \ref{thm_ContactMCG} (or the translation of Corollary \ref{cor_horizontal_twists}, which can be obtained in the same way) directly in contact geometry without having to go through the Anosov side.
\end{rem}

 \begin{proof}
 We start by proving the converse implication.  
  Let $f \in \cD(M,\xi)$ and let $\phi$ be the (Anosov) Reeb flow.  Then $f^{-1}\circ \varphi^t \circ f$ is a contact Anosov flow with contact structure $f_{\ast} \xi = \xi$. Thus by Theorem \ref{thm_isotopy_contact}, $f^{-1}\circ \varphi^t \circ f$ and $\varphi^t$ are isotopically equivalent. Let $h \colon M \to M$ be an orbit equivalence isotopic to identity between $f^{-1}\circ \varphi^t \circ f$ and $\varphi^t$, then $h\circ f$ is a self orbit equivalence of $\varphi^t$. Hence, by Theorems \ref{thm_oneJSJpiece} or \ref{thm_linear_graph} depending on the case, there exists $n$ such that $[f^n] = [(h\circ f)^n] \in D_\xi$.
  
  Now, for the second part, let $[f]$ be a class in $D_\xi$. Then, by Theorems \ref{thm_oneJSJpiece} or \ref{thm_linear_graph}, $\psi = f^{-1}\circ \varphi \circ f$ is isotopically equivalent to $\varphi$. Moreover, the contact structure of $\psi$ is $f_\ast\xi$, where $\xi$ is the contact structure of $\varphi$.
  By Theorem \ref{thm_isotopy_contact}, $f_\ast\xi$ and $\xi$ are isotopic, so there exists $g$ in the same isotopy class as $f$ such that $g_\ast\xi = \xi$. That is, $g\in \cD(M,\xi)$.
 \end{proof}
 
\bibliographystyle{amsalpha}
\bibliography{biblio_orb_eq}

 \end{document}